\documentclass[final,5p,times,twocolumn]{elsarticle}

\usepackage[ruled,vlined]{algorithm2e}
\usepackage{amsmath,amssymb,amsthm}
\usepackage{mathrsfs}
\usepackage{float}
\usepackage{color}
\usepackage{hyperref}
\usepackage{amsmath}
\usepackage{float}
\usepackage{graphicx}
\usepackage{subfigure}
\usepackage{epstopdf}
\usepackage{graphicx,color,xcolor}
\usepackage{caption}
\usepackage{lineno,hyperref}

\def\O{{\Omega}}

\def\p{{\partial}}

\def\div{{\mbox{div}}}

\journal{Journal}

\begin{document}

\begin{frontmatter}

\title{The Deep Learning Galerkin Method for the General Stokes Equations}

\author[d1]{Jian Li\corref{mycorrespondingauthor}}
\cortext[mycorrespondingauthor]{Corresponding author(Jian Li) email: jianli@sust.edu.cn\\Partly supported by the NSF of China 11771259 and Shaanxi special support plan for regional development of talents. Department of Mathematics, Shaanxi University of Science and Technology, Xi'an 710021, China.}
\author[d2]{Jing Yue}
\author[d1]{Wen Zhang}
\author[d3]{Wansuo Duan}

\address[d1]{Department of Mathematics, Shaanxi University of Science and Technology, Xi'an}
\address[d2]{School of electrical and control engineering, Shaanxi University of Science and Technology, Xi'an}
\address[d3]{Institute of Atmospheric Physics, Chinese Academy of Sciences, Beijing}

\begin{abstract}
The finite element method, finite difference method, finite volume method and spectral method have achieved great success in solving partial differential equations. However, the high accuracy of traditional numerical methods is at the cost of high efficiency. Especially in the face of high-dimensional problems, the traditional numerical methods are often not feasible in the subdivision of high-dimensional meshes and the differentiability and integrability of high-order terms. In deep learning, neural network can deal with high-dimensional problems by adding the number of layers or expanding the number of neurons. Compared with traditional numerical methods, it has great advantages. In this article, we consider the Deep Galerkin Method (DGM) for solving the general Stokes equations by using deep neural network without generating mesh grid. The DGM can reduce the computational complexity and achieve the competitive results. Here, depending on the $ L^{2} $ error we construct the objective function to control the performance of the approximation solution. Then, we prove the convergence of the objective function and the convergence of the neural network to the exact solution. Finally, the effectiveness of the proposed framework is demonstrated through some numerical experiments.
\end{abstract}

\begin{keyword}
general Stokes equations; Deep Galerkin Method; convergence; neural network; deep learning.
\MSC[2020] 00-01\sep  99-00
\end{keyword}

\end{frontmatter}

\section{Introduction}
Partial Differential Equations (PDEs) can mathematically model and describe certain objective laws in the fields of physical chemistry, finance, natural phenomenon and engineering technology \emph{et al}. However, most of them are difficult to obtain the analytical solution. Consequently, numerical methods such as finite element method have been flourishing in the past decades for modeling mechanics problems via solving  PDEs \cite{4}. Alternatively, the other methods, just like generalized finite element basis functions \cite{19} and construction of multiple difference schemes \cite{20} have broad applications in the same way. Although these methods are well used in PDEs and achieved good results, almost all of them have obvious drawbacks, complexity in general problems and no longer apparent since lots of mesh grid are generated especially for high dimensional problems. Besides, there has many large problems in computational fluid dynamics, such as uncertainty quantification, Bayesian inversion, data assimilation and constrained optimization of partial differential equations, which are considered to be very challenging because they require a large number of numerical solutions of the corresponding PDEs.

Inspired by machine learning,  the deep learning method can learn the parameters of neural network from the sampled data which can avoid mesh generation to some certain extent. Deep learning method has certain adaptability for unknown data, guarantee the high accuracy through training the models and currently gains a lot of interests for efficiently solving PDEs. It has been considered in various forms previously since the 1990s. Cellular Neural Network and Distributed Parameter Neural Network are used for one-dimensional PDEs \cite{21,22,23}. Apart from these, single layer chebyshev neural network \cite{26}, recurrent neural network and ansatz method \cite{27,28} can also solve the PDEs similarly. More generally, Sun \emph{et al.} \cite{24, 25} used Bernstein neural network and extreme learning machine to solve first and second order ordinary differential equations and elliptic PDEs.

Due to the rapid development of computer and gradient optimization methods in recent decades, many approaches for solving high-dimensional problems are actively proposed based on deep learning techniques.  Raissi {etc.} \cite{30} solved the Black-Scholes-Barenblatt and Hamilton-Jacobi-Bellman equations, both in 100 dimensions. Besides, they proposed and developed a typical method physics-informed neural networks which combines observed data with PDE models \cite{ Raissi1, Raissi2, Raissi3} in many problems \cite{Yang, Rao, Olivier, Lu, Fang, Pang}. As for higher-dimensional parametric PDEs system, Sirignano and Spiliopoulos proposed a continuous time stochastic gradient descent method \cite{31} and DGM \cite{32} for PDEs both in 200 dimensions. They applied a deep neural network instead of a linear combination of basis functions. Recently, Zhu \emph{et al.} and Xu \emph{et al.} {\cite{ Zhu1, Zhu2, Xu1, Xu2}} proposed Bayesian deep convolutional encoder-decoder network and the combination of genetic algorithm and adaptive method to solve the problems with high-dimensional random inputs and sparse noisy data. As we all know, there has mathematical guarantees called universal approximation theorems \cite{33} which stating that a single layer neural network could approximate many functions in Sobolev spaces. It still lack theoretical method to explain the effectiveness of multilayer neural networks.

In this paper, the DGM is first applied to solve the general $d$-dimentional incompressible Stokes problems, which is trained on batches of randomly sampled points satisfying the differential operator, initial condition and boundary condition without generating mesh grid. The optimal solution is obtained by using the stochastic gradient descent method instead of a linear combination of basic functions. In particular, this method overcomes the infeasibility and limitations of the traditional numerical methods especially for the high dimensional incompressible Stokes equations. Based on the objective function, the DGM numerically manifests the efficiency and flexibility. Moreover, we prove the convergence of the objective function and the convergence of the neural network and the exact solution.

The rest of the paper is developed into four sections. In the next section, we provide the preliminaries of methodology. In Section III, we prove the convergence of the objective function and the convergence of the neural network to the exact solution. In Section IV, numerical examples demonstrate the efficiency of the proposed framework and justify our theoretical analysis. Finally, we summarize our paper with a short discussion.

\section{Methodology}
Let $\Omega$ be a bounded, compact and open subset of $\mathbb{R^{\mathit{d}}}(d=2,3,...)$. With regular boundary $\partial\Omega\subset\mathbb{\mathbb{R^{\mathit{d\textrm{-1}}}}}$. We consider the general Stokes equations with Dirichlet boundary condition.
\begin{align}
\alpha u-\nu\nabla^{2}u+\nabla p&=f,\ \ \ \textrm{in}\ \Omega,\label{stokes-1}\\
\nabla\cdot u&=0,\ \ \ \textrm{in}\ \Omega,\label{stokes-2}\\
u&=g,\ \ \ \textrm{on}\ \partial\Omega,\label{stokes-3}
\end{align}
where $\alpha>0$ is a positive constant, $\nu$ denotes the viscosity coefficient, $u$ and $p$ represent velocity and pressure respectively, $\mathit{f}$ and $\mathit{g}$ are source terms. For notational brevity, we set $\overline{u}=(u,p)$ and define
\begin{equation}\label{6}
\mathcal{G}[\overline{u}]=\alpha u-\nu\nabla^{2}u+\nabla p-f.
\end{equation}
Here, we recall the classical Sobolev spaces

$$H^k(\Omega)=\Big\{\upsilon\in L^2(\Omega): D_w^{\alpha}\upsilon\in L^2(\Omega), \forall\alpha: \mid\alpha\mid\leq k\Big\},$$
$$H_0^k(\Omega)=\Big\{\upsilon\in H^k(\Omega): \upsilon\mid_{\partial\Omega}=0\Big\},$$
and their norm
$$\parallel\upsilon\parallel_k=\sqrt{(\upsilon,\upsilon)_k}=\bigg\{\sum_{\mid\alpha\mid=0}^{k}\int_{\Omega}(D_w^{\alpha}\upsilon)^{2}dx\bigg\}^{\frac{1}{2}},$$
where $k>0$ is a positive integer, $D_w^{\alpha}\upsilon$ is the generalized derivative of $\upsilon$, and $(\cdot,\cdot)$ represents the inner product.

In order to obtain a well-posedness of the general Stokes equations, we have the following result.
\newtheorem{lemma}{Lemma}[section]
\begin{lemma} \label{lemma1}
Assume that $\Omega$ is a bounded and connected open subset of $\mathbb{R}^d$ with a Lipschitz-continuous boundary $\Gamma$, $f\in[L^2(\Omega)]^d$ and $g\in [H^{1/2}(\Gamma)]^d$ such that
$$\int_\Gamma g\cdot \overrightarrow{n} ds=0,$$
there exists a unique pair $\overline{u}\in [H_{0}^{1}(\Omega)]^d\times L_0^2(\Omega)$ of the general Stokes equations (\ref{stokes-1})-(\ref{stokes-3}). Furthermore, we have \begin{align}
\big\| u\big\|_{2}+\big\| p\big\|_{1}\leq C(\big\| f\big\|_{-1}+\big\| g\big\|_{3/2,\partial\Omega}).
\end{align}
\end{lemma}

Generally, assuming that $\overline{U}=\big(U(\mathbf{x;\theta_{1}}),P(\mathbf{x;\theta_{2}})\big)$ is the neural network solution to the general Stokes equations (\ref{stokes-1})-(\ref{stokes-3}), $\mathbf{\theta_{1}}$ and $\mathbf{\theta_{2}}$ are the stacked components of the neural network's parameters $\theta$ for velocity and pressure respectively. Define the objective function
\begin{equation}\label{7}
\begin{split}
J(\overline{U})&=\Big\|\mathcal{G}[\overline{U}](x;\theta)-\mathcal{G}[\overline{u}](x)\Big\|_{0, \Omega, \omega_{1}}^{2}+\Big\|\nabla\cdot {U}(x;\theta_1)\Big\|_{0, \Omega,\omega_{1}}^{2}\\
&+\Big\|U(x;\theta_1)-g(x)\Big\|_{0, \partial\Omega, \omega_{2}}^{2}.
\end{split}
\end{equation}

It should be noted that $J(\overline{U})$ can measure how well the approximate solution satisfies differential operator, divergence condition and boundary condition. Notice that
$$\big\| f(y)\big\|_{0,\mathcal{Y},\omega}=\int_{\mathcal{Y}}\big| f(y)\big|^{2}\omega(y)dy,$$
where $\omega(y)$ is the probability density of $y$ in $\mathcal{Y}$.

Our goal is to find the parameters $\theta$ such that $\overline{U}$ minimizes the objective function $J(\overline{U})$. Especially, if $J(\overline{U})=0$ then $\overline{U}$ is the solution to the general Stokes equations (\ref{stokes-1})-(\ref{stokes-3}). However, it is computationally infeasible to estimate $\theta$ by directly minimizing $J(\overline{U})$ when integrated over a higher dimensional region. Here, we apply a sequence of random sampled points from $\Omega$ and $\partial\Omega$ to avoid forming mesh grid. The algorithm of the DGM for the general Stokes equations are presented as Algorithm 1.

\begin{algorithm}
\caption{Deep Learning Galerkin Method}
\KwIn{$s_{n}=(x_{n},r_{n})$, Max Iterations $M$,  learning rate $\alpha_n$ }
\KwOut{$\theta_{n+1}$}
Randomly generated sample points $s_{n}=(x_{n},r_{n})$; \\
Initialize the parameters  $\theta$\;
\While{iterations $\leq M$}{
read current;
      \begin{align*}
  G(\theta_{n},s_{n})&=\Big(\mathcal{G}[\overline{U}](x_{n};\theta)\Big)^{2}\\
  &+\big(\nabla\cdot U(x_{n};\theta_1)\big)^{2}+\Big(U(r_{n};\theta_1)-g(x)\Big)^{2}
  \end{align*}\\
 \textbf{ and}\\
   $$\theta_{n+1}=\theta_{n}-\alpha_{n}\nabla _{\theta}G(\theta_{n},s_{n}).$$\\
    \eIf{$\underset{n\rightarrow\infty}{\lim}\| \nabla_{\theta}G(\theta_{n},s_{n})\|=0$}{
        return the parameters $\theta_{n+1}$\;
    }{
        go back to the beginning of current section\;
    }
}
\end{algorithm}

In this process, the ``learning rate" $\alpha_{n}\in(0,1)$ decreases as $n\rightarrow\infty$. The term $\nabla_{\theta}G(\theta_{n},s_{n})$ is unbiased estimate of $\nabla_{\theta}J\big(\overline{U}(\cdot;\theta_{n})\big)$ because we can estimate the population parameters by sample mathematical expectations such as
\begin{equation}\label{ppp}
\mathbb{E}\Big[\nabla_{\theta}G(\theta_{n},s_{n})\mid \theta_{n}\Big]=\nabla_{\theta}J\big(\overline{U}(\cdot;\theta_{n})\big).
\end{equation}

In order to illustrate more vividly, the flowchart displayed in Figure \ref{Fig.2}.

\section{Convergence}
\vspace{0.3cm} Undoubtedly, the objective function $J(\overline{U})$ can measure how well the neural network $\overline{U}$ satisfies the differential operator, boundary condition and divergence condition. As we known from \cite{33}, if there is only one hidden layer and one output, then the set of all functions implemented by such a network with $m$ and $n$ hidden units for velocity and pressure are
\begin{equation}\nonumber
\begin{split}
&[\mathfrak{C}_{u}^{\mathrm{m}}(\varphi)]^{d}\\
&=\Big\{\Phi(x): \mathbb{R^{\mathrm{d}}\mapsto\mathbb{R^{\mathrm{d}}}\Big|}\Phi_{\ell}(x)={\sum\limits_{i=1}^{m}}\beta_{i}\varphi\big({\sum\limits_{j=1}^{d}}\sigma_{ji}x_{j}+c_{i}\big)\Big\},
\end{split}
\end{equation}
and
\begin{equation}\nonumber
\begin{split}
\mathfrak{C}_{p}^{\mathrm{n}}(\psi)=
\Big\{\Psi(x): \mathbb{R^{\mathrm{d}}\mapsto\mathbb{R}\Big|}\Psi(x)={\sum\limits_{i=1}^{n}}\beta'_{i}\psi\big({\sum\limits_{j=1}^{d}}\sigma'_{ji}x_{j}+c'_{i}\big)\Big\},\\
\end{split}
\end{equation}
where $\ell=1,2,\ldots,d, \ \Phi(x)=\big(\Phi_1(x),\Phi_2(x),\cdots ,\Phi_d(x)\big)$, $\varphi$ and $\psi$ are the shared activation functions of the hidden units in $\mathcal{C}^2(\Omega)$, bounded and non-constant. $x_j$ is input, $\beta_i,\beta'_i,\sigma_{ji}$ and $\sigma'_{ji}$ are weights, $c_i$ and $c'_i$ are thresholds of neural network.

More generally, we still use the similar notation
$$[\mathfrak{C}_{u}(\varphi)]^{d}\times\mathfrak{C}_{p}(\psi)$$
for the multilayer neural networks with an arbitrarily large number of hidden units $m$ and $n$ respectively. In particular, we use the same parameters $\beta_i,\sigma_{ji},c_i$ and common activation function $\varphi$ in each dimension of $[\mathfrak{C}_{u}^m(\varphi)]^{d}$. The parameters can be written as follows
\begin{equation}\label{9}
\begin{aligned}
\theta_{1}^{\ell}&=(\beta_{1},\cdots,\beta_{m},\sigma_{11},\cdots,\sigma_{dm},c_{1},\cdots,c_{m}),\\
\theta_{2}&=(\beta'_{1},\cdots,\beta'_{n},\sigma'_{11},\cdots,\sigma'_{dn},c'_{1},\cdots,c'_{n}),
\end{aligned}
\end{equation}
where $\ell=1,2,\ldots,d, \theta_{1}\in \mathbb{R}^{(2+d)md}$ and $\theta_{2}\in \mathbb{R}^{(2+d)n}$.

In this section, we show that the neural network $\overline{U}^n$ with $n$ hidden units for $U$ and $P$ satisfies the differential operator, boundary condition and divergence condition arbitrarily well for sufficiently large $n$. Moreover, we prove that there exists $\overline{U}^n\in[\mathfrak{C}_{u}^n(\varphi)]^{d}\times\mathfrak{C}_{p}^n(\psi)$ such that $J(\overline{U}^n)\rightarrow0$ as $n\rightarrow\infty$. Another significant consideration, we give the convergence of $\overline{U}^{n}\rightarrow\overline{u}$ as $n\rightarrow\infty$ where $\overline{u}$ is the exact solution to the general Stokes equations (\ref{stokes-1})-(\ref{stokes-3}).
\begin{figure}[ht]
\centering
\includegraphics[scale=0.4]{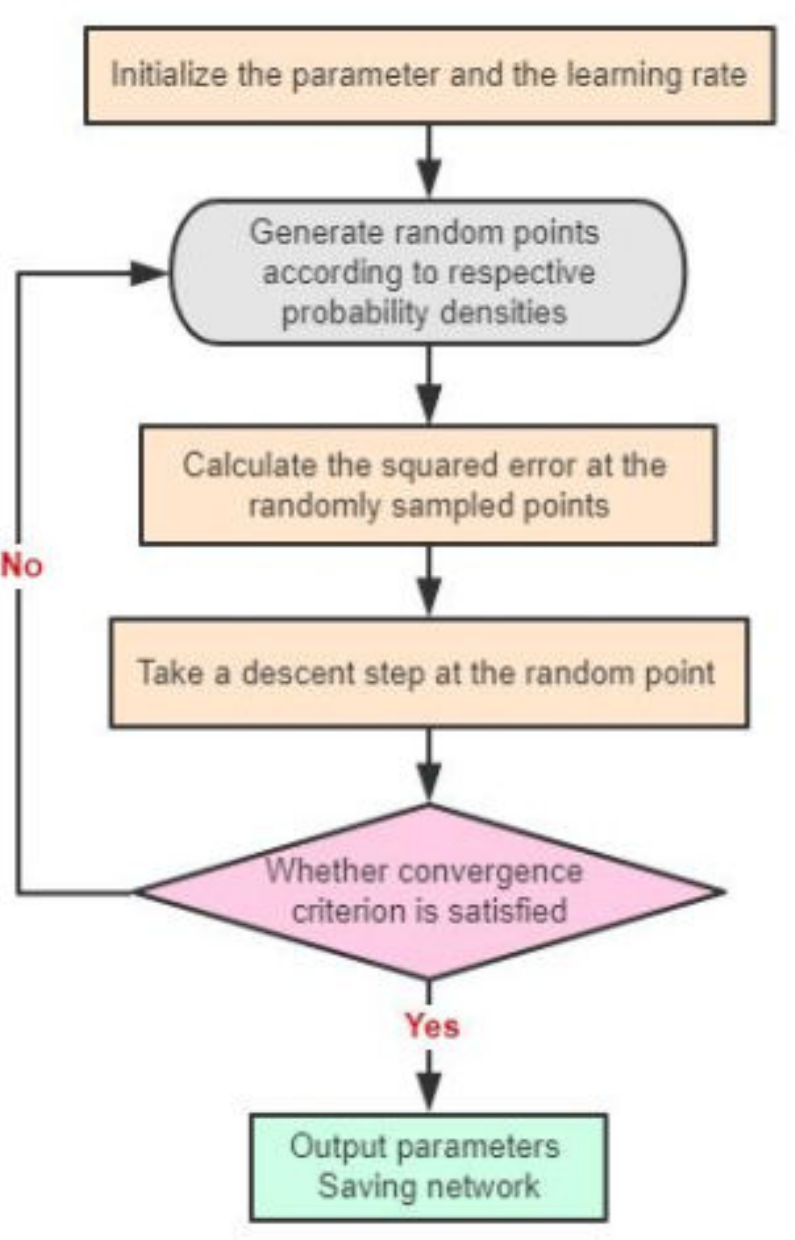}
\caption{Flowchart of the DGM for the general Stokes equations.}
\label{Fig.2}
\end{figure}
\subsection{Convergence of the objective function $J(\overline{U})$}
A particularly important processing, we use the multilayer feed forward networks $\overline{U}$ to universally approximate the solution to the general Stokes equations. Certainly, the neural network $\overline{U}$ can make the objective function $J(\overline{U})$ arbitrarily small. Thus, using the results of \cite{33} and the following lemma, we obtain the convergence of the objective function $J(\overline{U})$. First, we give the following assumption.
\begin{lemma}\label{lemma2}
Assume that $\nabla u(x),\triangle u(x)$ and $\nabla p(x)$ are locally Lipschitz with Lipschitz coefficient that they have at most polynomial growth on $u(x)$ and $p(x)$. Then, for some constants $0\leq q_{i}\leq\infty(i=1,2,3,4)$ we have
\begin{equation}\label{18}
\mid\triangle U-\triangle u\mid\leq(\mid\nabla U\mid^{q_{1}/2}+\mid\nabla u\mid^{q_{2}/2})\mid\nabla U-\nabla u\mid,
\end{equation}
\begin{equation}\label{19}
\mid\nabla P-\nabla p\mid\leq(\mid P\mid^{q_{3}/2}+\mid p\mid^{q_{4}/2})\mid P-p\mid.
\end{equation}
\end{lemma}
\newtheorem{thm}{\bf Theorem}[section]
\begin{thm}\label{thm1}
Under the assumption of Lemma \ref{lemma2}, there exists a neural network $\overline{U}\in[\mathfrak{C}_{u}(\varphi)]^d\times \mathfrak{C}_{p}(\psi)$, satisfying
\begin{equation}\label{th1}
J(\overline{U})\leq C\epsilon, \ \forall \epsilon>0,
\end{equation}
where $C$ depends on the data $\{\Omega, \alpha, \nu, \omega_{1}, \omega_{2}, f\}$.
\end{thm}
\begin{proof}
By Theorem 3 of \cite{33}, we can conclude that there exists $\overline{U}\in[\mathfrak{C}_{u}(\varphi)]^{d}\times\mathfrak{C}_{p}(\psi)$ which are uniformly 2-dense on compacts of $\mathcal{C} ^{2}(\bar{\Omega})\times\mathcal{C} ^{1}(\bar{\Omega})$. It means that for $\overline{u}\in \mathcal{C} ^{2}(\bar{\Omega})\times\mathcal{C} ^{1}(\bar{\Omega}), \forall\epsilon >0$, it follows that
\begin{equation}\label{16}
\underset{a\leq2}{max}\underset{x\in\Omega}{sup}\mid \partial_{x}^{a}U(x)-\partial_{x}^{a}u(x)\mid<\epsilon,
\end{equation}
\begin{equation}\label{17}
\underset{x\in\Omega}{sup}\mid P(x)-p(x)\mid<\epsilon.
\end{equation}

According to the Lemma \ref{lemma2}, using the H$\ddot{o}$lder inequality and Young inequality, setting $r_{1}$ and $r_{2}$ are conjugate numbers such that $\frac{1}{r_{1}}+\frac{1}{r_{2}}=1$, we find that
\begin{equation}
\begin{aligned}\label{20}
&\int_{\Omega}\mid\triangle U-\triangle u\mid^{2}d\omega_{1}(x) \\
\leq  &\int_{\Omega}\Big(\mid\nabla U\mid^{q_{1}}+\mid\nabla u\mid^{q_{2}}\Big)\Big(\nabla U-\nabla u\Big)^{2}d\omega_{1}(x)\\
\leq  &\Big[\int_{\Omega}\Big(\mid\nabla U\mid^{q_{1}}+\mid\nabla u\mid^{q_{2}}\Big)^{r_{1}}d\omega_{1}(x)\Big]^{1/r_{1}}\\
&\times\Big[\int_{\Omega}\Big(\nabla U-\nabla u\Big)^{2r_{2}}d\omega_{1}(x)\Big]^{1/r_{2}} \\
\leq  &\Big[\int_{\Omega}\Big(\mid\nabla U-\nabla u\mid^{q_{1}}+\mid\nabla u\mid^{q_{1}\vee q_{2}}\Big)^{r_{1}}d\omega_{1}(x)\Big]^{1/r_{1}}\\
&\times\Big[\int_{\Omega}\Big(\nabla U-\nabla u\Big)^{2r_{2}}d\omega_{1}(x)\Big]^{1/r_{2}}\\
\leq  &C\epsilon^{2},
\end{aligned}
\end{equation}
where $q_{1}\vee q_{2}=max\{q_{1}, q_{2}\}$.

Similarly,
\begin{equation}
\begin{aligned}\label{21}
&\int_{\Omega}\mid\nabla P-\nabla p\mid^{2}d\omega_{1}(x) \\
\leq  &\int_{\Omega}\Big(\mid P\mid^{q_{3}}+\mid p\mid^{q_{4}}\Big)\Big(P-p\Big)^{2}d\omega_{1}(x)\\
\leq  &\Big[\int_{\Omega}\Big(\mid P\mid^{q_{3}}+\mid p\mid^{q_{4}}\Big)^{r_{3}}d\omega_{1}(x)\Big]^{1/r_{3}}\\
&\times\Big[\int_{\Omega}
(P-p)\mid^{2r_{4}}d\omega_{1}(x)\Big]^{1/r_{4}}\\
\leq  &\Big[\int_{\Omega}\Big(\mid P-p\mid^{q_{3}}+\mid p\mid^{q_{3}\vee q_{4}}\Big)^{r_{3}}d\omega_{1}(x)\Big]^{1/r_{3}}\\
&\times\Big[\int_{\Omega}
\Big(P-p\Big)^{2r_{4}}d\omega_{1}(x)\Big]^{1/r_{4}}\\
\leq &C\epsilon^{2},
\end{aligned}
\end{equation}
where $\frac{1}{r_{3}}+\frac{1}{r_{4}}=1$ and $q_{3}\vee q_{4}=max\{q_{3}, q_{4}\}$.

For the boundary condition, we have
\begin{align}\label{22}
&\int_{\partial\Omega}\mid U-u\mid^{2}d\omega_{2}(x)\leq C\epsilon^{2}.
\end{align}

Thanks to (\ref{20})-(\ref{22}), we obtain
\begin{equation}
\begin{aligned}\label{23}
J(\overline{U})=&\big\|\mathcal{G}[\overline{U}](x;\theta)-\mathcal{G}[\overline{u}](x)\big\|_{\Omega,\omega_{1}}^{2}\\
&+\big\|\nabla\cdot {U}(x;\theta)\big\|_{\Omega,\omega_{1}}^{2}+\big\|U(x;\theta)-g(x)\big\|_{\partial\Omega,\omega_{2}}^{2}\\
=&\big\|\mathcal{G}[\overline{U}](x;\theta)\big\|_{\Omega,\omega_{1}}^{2}\\
&+\big\|\nabla\cdot {U}(x;\theta)\big\|_{\Omega,\omega_{1}}^{2}+\big\|U(x;\theta)-g(x)\big\|_{\partial\Omega,\omega_{2}}^{2}\\
=&\int_{\Omega}\mid\triangle U-\triangle u\mid^{2}d\omega_{1}(x)+\int_{\Omega}\mid\nabla P-\nabla p\mid^{2}d\omega_{1}(x)\\
&+\int_{\Omega}\mid\alpha U-\alpha u\mid^{2}d\omega_{1}(x)+\int_{\Omega}\mid\nabla\cdot u\mid^{2}d\omega_{1}(x)\\
&+\int_{\Omega}\mid\nabla\cdot( U-u)\mid^{2}d\omega_{1}(x)+\int_{\partial\Omega}\mid U-u\mid^{2}d\omega_{2}(x)\\
\leq  &C\epsilon^{2},
\end{aligned}
\end{equation}
which implies (\ref{th1}).
\end{proof}
\subsection{Convergence of the neural network to the general Stokes solution}
We have discussed the convergence of the objective function $J(\overline{U})$ in the last subsection. Next we give the convergence of the neural network $\overline{U}^n$ to the exact solution $\overline{u}$ for the general Stokes equations with homogeneous boundary condition
\begin{align}
\alpha u-\nu\nabla^{2}u+\nabla p&=f,\ \ \ \textrm{in}\ \Omega,\label{hstokes-1}\\
\nabla\cdot u&=0,\ \ \ \textrm{in}\ \Omega,\label{hstokes-2}\\
u&=0,\ \ \ \textrm{on}\ \partial\Omega.\label{hstokes-3}
\end{align}

Recall the form of the objective function with
$$J(\overline{U})=\|\mathcal{G}[\overline{U}]\|_{0,\Omega}^{2}+\|\nabla\cdot {U}\|_{0,\Omega}^{2}+\|U\|_{0,\partial\Omega}^{2}. $$

By Theorem \ref{thm1}, we obtain
$$ J(\overline{U}^{n})\rightarrow 0~~\text{as}~~n\rightarrow \infty. $$

Furthermore, each neural network $\overline{U}^{n}=(U^n,P^n)$ satisfies the following equations
\begin{align}
\mathcal{G}[\overline{U}^{n}]&=h^{n},\ \ \textrm{in}\ \Omega,\label{ustokes-1}\\
\nabla\cdot {U}^{n}&=0,\ \ \ \ \textrm{in}\ \Omega,\label{ustokes-2}\\
{U}^{n}&=g^{n},\ \ \textrm{on}\ \partial\Omega,\label{ustokes-3}
\end{align}
for some $h^{n}$ and $g^{n}$ such that
\begin{equation}\label{25}
\begin{aligned}
\|h^{n}\|_{0,\O}^{2}+\|g^{n}\|_{0,\p\O}^{2} \rightarrow 0~~\text{as}~~n\rightarrow \infty.
\end{aligned}
\end{equation}

In this subsection, we do not explore more discussions on inhomogeneous problems since the inhomogeneous problems can be solved by the corresponding homogeneous method (See Section 4 of Chapter V in \cite{34} or Chapter 8 of \cite{35} for details). For convenience, we provide a theorem to guarantee the convergence of the neural network $\overline{U}^n$ and the exact solution $\overline{u}$ to the equations (\ref{hstokes-1})-(\ref{hstokes-3}).

\begin{thm}\label{thm2}
Under the assumptions of Lemma \ref{lemma2}, Theorem \ref{thm1} and (\ref{25}), the neural network $U^{n}$ can converge strongly to $u$ in $L^{2}(\Omega)$, and the $P^{n}$ converges strongly to $p$ in $H^{-1}(\Omega)$. In addition, if the sequences $\{U^{n}\}_{n\in\mathbb{N}}$ and $\{P^{n}\}_{n\in\mathbb{N}}$ are uniformly bounded and equicontinuous in $\Omega$, they can converge to $u$ and $p$ uniformly in $\Omega$ respectively.
\end{thm}
\begin{proof}
The existence and uniqueness for the solution of (\ref{hstokes-1})-(\ref{hstokes-3}) are proved by the Saddle point theorem (See Lemma \ref{lemma1}). Note that $\widehat{\overline{U}}^{n}$ satisfies the equations (\ref{ustokes-1})-(\ref{ustokes-3}) with $g^{n}=0$. Firstly, we know that the variational formulation of the equations (\ref{ustokes-1})-(\ref{ustokes-3}) is to find $(\widehat{U}^n,\widehat{P}^n)\in [\mathfrak{C}_{u}^{n}(\varphi)]^d\times \mathfrak{C}_{p}^{n}(\psi)$ for $\forall (\widehat{V},\widehat{Q})\in [\mathfrak{C}_{u}^{n}(\varphi)]^d\times \mathfrak{C}_{p}^{n}(\psi)$ such that
\begin{equation}\label{27}
\begin{aligned}
&\alpha(\widehat{U}^{n},\widehat{V})+\nu(\nabla\widehat{U}^{n},\nabla\widehat{V})+(\nabla\widehat{P}^{n},\widehat{V})+(\nabla\cdot\widehat{U}^{n},\widehat{Q})\\
&=(f,\widehat{V})+(h^{n},\widehat{V}),
\end{aligned}
\end{equation}
in addition, we have
\begin{equation}\label{227}
(\nabla\widehat{P}^{n},\widehat{V})=-(\nabla\cdot\widehat{V},\widehat{P}^{n}).
\end{equation}

Taking $\widehat{V}=\widehat{U}^{n}$ and $\widehat{Q}=\widehat{P}^{n}$ in (\ref{27}), it follows that
\begin{equation}\label{28}
\begin{aligned}
\alpha(\widehat{U}^{n},\widehat{U}^{n})+\nu(\nabla\widehat{U}^{n},\nabla\widehat{U}^{n})&=(f,\widehat{U}^{n})+(h^{n},\widehat{U}^{n}).
\end{aligned}
\end{equation}

Using the definition of the $H^{1}$ norm, (\ref{25}) and setting $\alpha_{0}=min\{\alpha, \nu\}$, we obtain
\begin{equation}\label{29}
\begin{aligned}
\alpha_{0}\parallel\widehat{U}^{n}\parallel_{1,\O}&\leq C(\parallel f\parallel_{0,\O}+\parallel h^n\parallel_{0,\O})\\
&\leq C\parallel f\parallel_{0,\O}.
\end{aligned}
\end{equation}

The convergence of $\widehat{\overline{U}}^{n}$ and $\overline{U}^{n}$ is desirable to discuss yet. By using the uniformly boundedness of $\widehat{U}^{n}$, we can extract a subsequence $\{\widehat{U}^{n}\}_{n\in\mathbb{N}}$ of $\widehat{U}^{n}$ which can converge weakly in $H^{1}(\Omega)$. Due to the compact embedding $H^{1}(\Omega)\hookrightarrow L^{2}(\Omega)$,
we have $\underset{n\rightarrow\infty}{\lim}\parallel \widehat{U}^{n}-u\parallel_{0,\O}=0$.

Nevertheless, we remain to discuss $\underset{n\rightarrow\infty}{\lim}\parallel U^{n}-\widehat{U}^{n}\parallel_{0,\O}=0$, where $U^{n}$ and $\widehat{U}^{n}$ satisfy (\ref{ustokes-1})-(\ref{ustokes-3}) with homogeneous and inhomogeneous boundary respectively.  Afterwards, since $\underset{n\rightarrow\infty}{\lim}\|g^{n}\|_{0,\p\O}=0$, $U^{n}$ converges to zero at least along a subsequence on the boundary. Besides, it will be identical with $\widehat{U}^{n}$ almost everywhere. Indeed, define $F_{n}=|U^{n}-\widehat{U}^{n}|^2$. $\big\{ F_{n}(x)\big\}_{n\in\mathbb{N}}$ is uniformly bounded in $L^2(\Omega)$ by the reason of the uniformly boundedness of $\big\{U^{n}\big\}_{n\in\mathbb{N}}$ and $\big\{ \widehat{U}^{n}\big\}_{n\in\mathbb{N}}$. In addition, $\big\{ F_{n}(x)\big\}_{n\in\mathbb{N}}$ can be integrated on domain $\bar{\Omega}$ and converges to zero almost everywhere. By the definition of $F_{n}$, the uniformly boundedness and equicontinuity of $\big\{U^{n}\big\}_{n\in\mathbb{N}}$ and $\big\{\widehat{U}^{n}\big\}_{n\in\mathbb{N}}$, for $\forall x,y\in \bar{\O}$, $\forall\epsilon'>0$, $\exists\delta>0$, if $|x-y|<\delta$, there holds that
\begin{equation}
\begin{aligned}
&\Big| F_{n}(x)- F_{n}(y)\Big|\\
= &\Big|\big|U^{n}(x)-\widehat{U}^{n}(x)\big|^2-\big|U^{n}(y)-\widehat{U}^{n}(y)\big|^2\Big|\\
=& \Big|\big|U^{n}(x)-\widehat{U}^{n}(x)\big|+\big|U^{n}(y)-\widehat{U}^{n}(y)\big|\Big|\\
&\times\Big|\big|U^{n}(x)-\widehat{U}^{n}(x)\big|-\big|U^{n}(y)-\widehat{U}^{n}(y)\big|\Big|\\
\leq & \Big|\big|U^{n}(x)-\widehat{U}^{n}(x)\big|+\big|U^{n}(y)-\widehat{U}^{n}(y)\big|\Big|\\
&\times\Big|\big|U^{n}(x)-\widehat{U}^{n}(x)-U^{n}(y)-\widehat{U}^{n}(y)\big|\Big|\\
\leq & \Big|\big|U^{n}(x)-\widehat{U}^{n}(x)\big|+\big|U^{n}(y)-\widehat{U}^{n}(y)\big|\Big|\\
&\times\Big|\big|U^{n}(x)-U^{n}(y)\big|+\big|\widehat{U}^{n}(x)-\widehat{U}^{n}(y)\big|\Big|\\
<&C\epsilon',
\end{aligned}
\end{equation}
where $\epsilon'>0$ is an arbitrarily small constant. In conclusion, $\big\{ F_{n}(x)\big\}_{n\in\mathbb{N}}$ is equicontinous. Based on the above preparation, we can obtain $\underset{n\rightarrow\infty}{\lim}\parallel U^{n}-\widehat{U}^{n}\parallel_{0,\O}=0$ by using Vitali's theorem. Thus through a triangle inequality there holds that
\begin{equation}
\begin{aligned}
&\underset{n\rightarrow\infty}{\lim}\parallel U^{n}-u\parallel_{0,\O}\\
\leq &\underset{n\rightarrow\infty}{\lim}\parallel U^{n}-\widehat{U}^{n}\parallel_{0,\O}+\underset{n\rightarrow\infty}{\lim}\parallel \widehat{U}^{n}-u\parallel_{0,\O}\\
=&0
\end{aligned}
\end{equation}
since $\{\widehat{U}^{n}\}_{n\in \mathbb{N}}$ strongly converges to $u$ in $L^2(\Omega)$.

In order to study the pressure of the general Stokes problem, we define
\begin{equation}\label{30}
\begin{aligned}
L(\widehat{v})&=(f,\widehat{v})+(h^{n},\widehat{v})-\alpha(\widehat{U}^{n},\widehat{v})-\nu(\nabla\widehat{U}^{n},\nabla\widehat{v})\\
&=0,
\end{aligned}
\end{equation}
where $\widehat{v}\in [\mathfrak{C}_{u}^{n}(\varphi)]^d\cap [H_0^1(\Omega)]^d$. Then
$$<L, \widehat{v}>=0, \ \forall\widehat{v}\in[\mathfrak{C}_{u}^{n}(\varphi)]^d\cap [H_0^1(\Omega)]^d,$$
where $<\cdot,\cdot>$ stands for the duality pairing between $[\mathfrak{C}_{u}^{n}(\varphi)]^d\cap [H_0^1(\Omega)]^d$ and it's dual space.

In addition, there exists $\widehat{P}^{n}\in \mathfrak{C}_{p}^{n}(\psi)\cap L^{2}(\Omega),$ for $\forall \widehat{v}\in[\mathfrak{C}_{u}^{n}(\varphi)]^d\cap [H_0^1(\Omega)]^d$ such that
$$<L, \widehat{v}>=\int_{\Omega}\widehat{P}^{n}\div\widehat{v}dx=-(\widehat{P}^{n},\div\widehat{v})=d(\widehat{v},\widehat{P}^{n}).$$

Namely,
\begin{align}\label{lv}
d(\widehat{v},\widehat{P}^{n})=(f,\widehat{v})+(h^{n},\widehat{v})-\alpha(\widehat{U}^{n},\widehat{v})-\nu(\nabla\widehat{U}^{n},\nabla\widehat{v}).
\end{align}

What's more, as in Theorem $3.3$ of \cite{36}, we find that
$$\nabla\widehat{U}^{n}\rightarrow \nabla u\ \ almost\ \  everywhere\ \ in \ \ \Omega,$$
which concludes that $\widehat{P}^{n}$ can converge weakly to $p$ since $d(\widehat{v},\widehat{P}^{n})\rightharpoonup d(\widehat{v},p)$. Applying the same approach as for the strong convergence of $\{\widehat{U}^{n}\}_{n\in \mathbb{N}}$ to $u$ in $L^{2}(\Omega)$. Consequently, due to the compact embedding $L^{2}(\Omega)\hookrightarrow H^{-1}(\Omega)$, we can obtain
$$\underset{n\rightarrow\infty}{\lim}\parallel \widehat{P}^{n}-p\parallel_{-1,\O}=0.$$

Using a triangle inequality, it follows that
\begin{equation}
\begin{aligned}
 &\underset{n\rightarrow\infty}{\lim}\parallel P^{n}-p\parallel_{-1,\O}\\
\leq &\underset{n\rightarrow\infty}{\lim}\parallel P^{n}-\widehat{P}^{n}\parallel_{-1,\O}+\underset{n\rightarrow\infty}{\lim}\parallel \widehat{P}^{n}-p\parallel_{-1,\O}\\
=&0.
\end{aligned}
\end{equation}

For all these reasons, $\{U^{n}\}_{n\in \mathbb{N}}$ can converge strongly to $u$ in $L^{2}(\Omega)$, $\{P^{n}\}_{n\in \mathbb{N}}$ converges strongly to $p$ in $H^{-1}(\Omega)$. Noting that $\{U^{n}\}_{n\in \mathbb{N}}$ and $\{P^{n}\}_{n\in \mathbb{N}}$ are uniformly bounded and equicontinuous in $\Omega$, we can conclude that   $\{U^{n}\}_{n\in \mathbb{N}}$ and $\{P^{n}\}_{n\in \mathbb{N}}$ converge uniformly to $u$ and $p$ by the well known Arzel$\grave{a}$-Ascoli theorem.
\end{proof}
\section{Numerical Experiments}

In this section, we apply the DGM to solve the general Stokes problems in both 2D and 3D case. The experimental results show the high efficiency and precision of the DGM. Our numerical experiments are based on Tensorflow \cite{37} and the configuration of the computer is 64-bit Intel Xeon Silver 4116 (2 processors). In the numerical simulation, we utilize six different architectures to train the neural network and set the same number of network layers to solve $U$ and $P$ simultaneously. These architectures include one to six hidden layers respectively, and 16 units on each hidden layer (Denoted as ARCH 1-6). We apply ARCH 1-3 in 2D case and apply ARCH 1-6 in 3D case. The datasets in 2D case contain 1000, 2000, 4000 and 8000 samples respectively (See Figure \ref{Fig.5}, denoted as 2D-DS 1-4). And in 3D case, the datasets contain 1200, 2400, 4800 and 9600 samples respectively (See Figure \ref{Fig.3d1}, denoted as 3D-DS 1-4).
\begin{figure}[ht]
\centering
\includegraphics[scale=0.4]{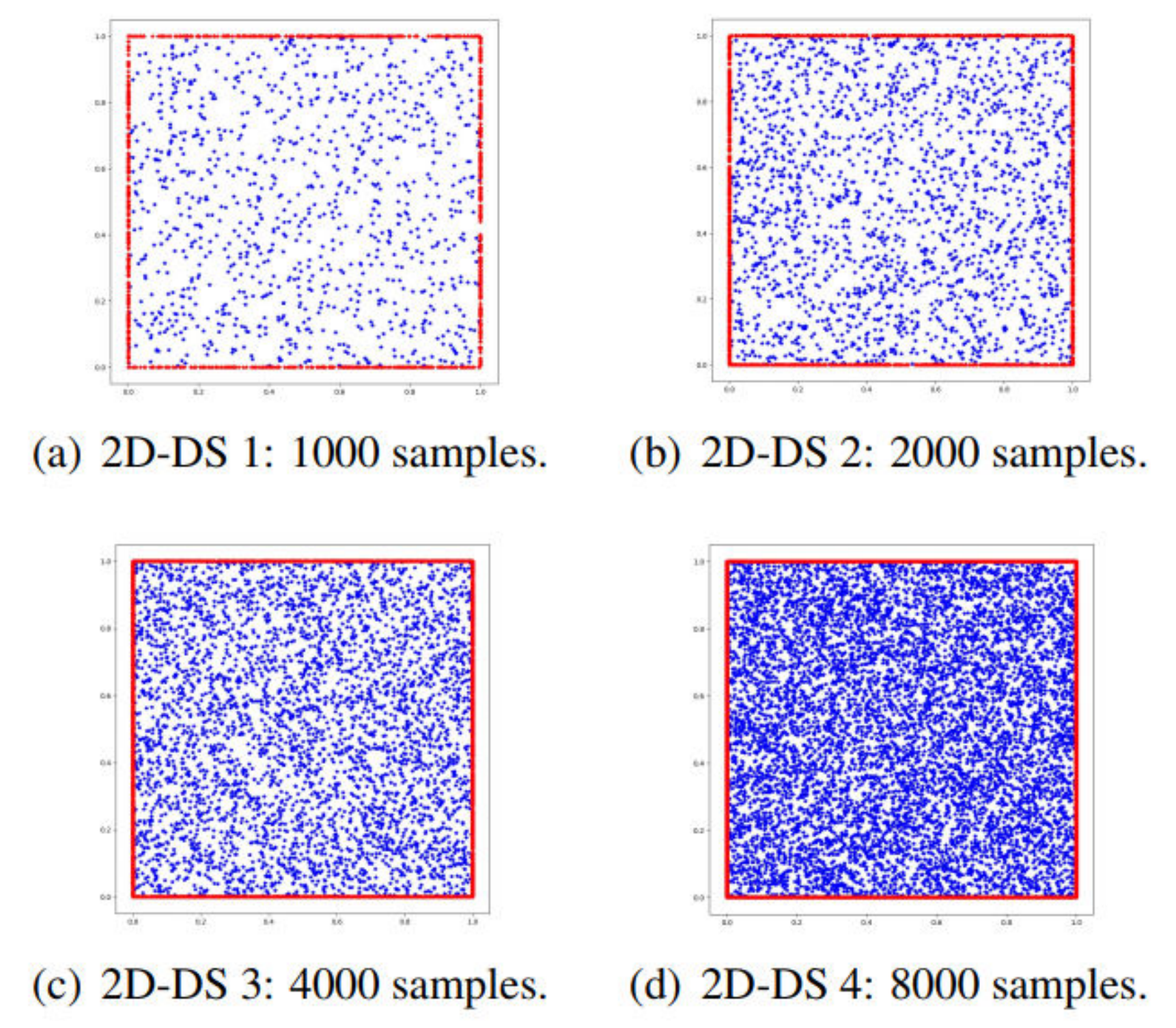}
\caption{The datasets in 2D case.}
\label{Fig.5}
\end{figure}

\begin{figure}[htbp]
\centering
\includegraphics[scale=0.4]{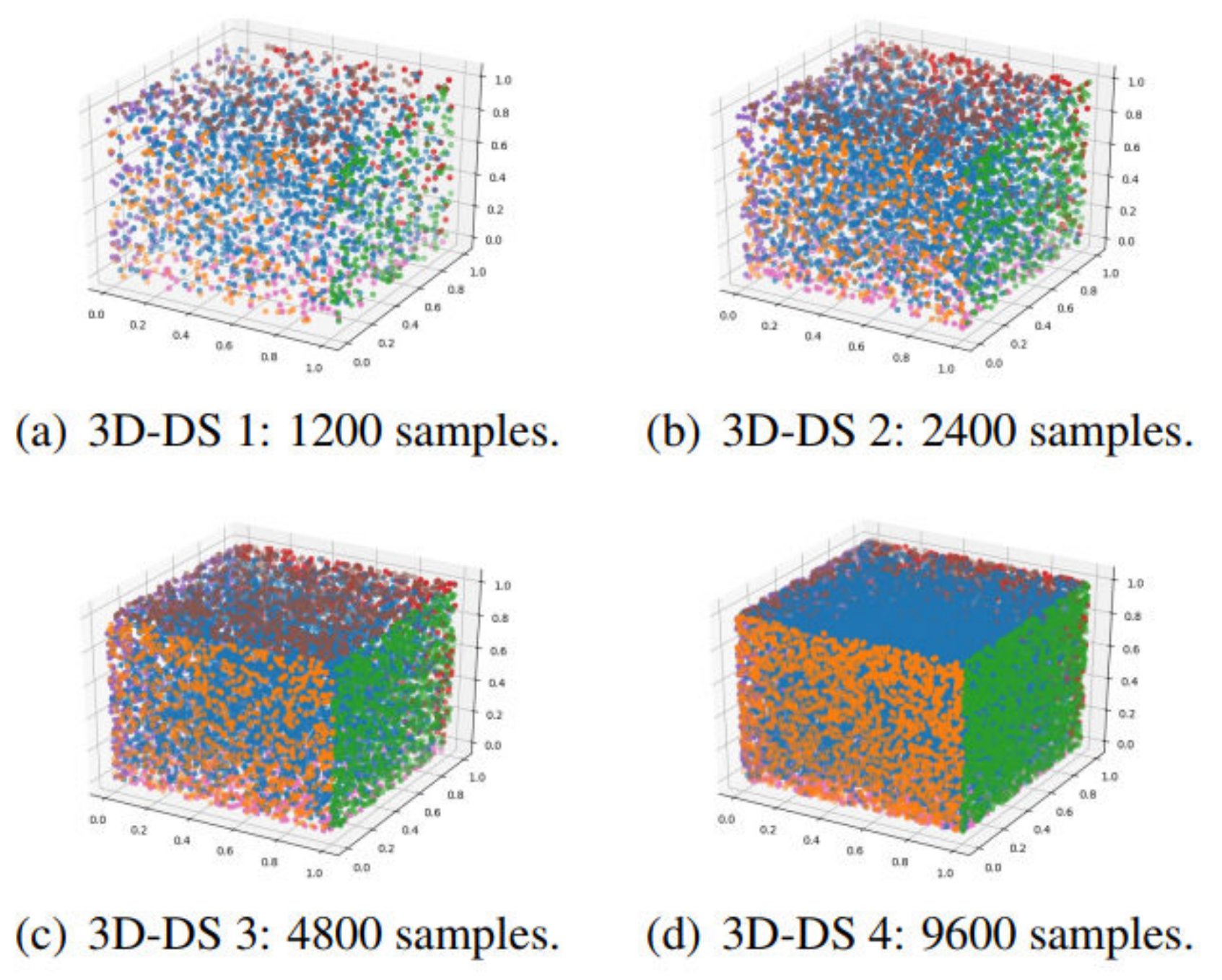}
\caption{The datasets in 3D case.}
\label{Fig.3d1}
\end{figure}
\subsection{The Stokes equations}

In this subsection, we first consider the Stokes equations with homogeneous boundary condition in both 2D and 3D cases. Set $\nu=0.025$ and $\alpha=0$ in equations (\ref{stokes-1}) - (\ref{stokes-3}). Use the following 2D and 3D exact solutions,
\begin{equation}\label{11}
\begin{aligned}
u_{1}(x_{1},x_{2})&=2sin(\pi x_{1})^{2}sin(\pi x_{2})cos(\pi x_{2})\pi,\\
u_{2}(x_{1},x_{2})&=-2sin(\pi x_{1})sin(\pi x_{2})^{2}cos(\pi x_{1})\pi,\\
p(x_{1},x_{2})&=cos(\pi x_{1})cos(\pi x_{2}),
\end{aligned}
\end{equation}
in $\Omega=(0,1)^2$ and
\begin{equation}
\begin{aligned}\label{3dequation1}
&u_{1}(x,y,z)\\&=sin(\pi x)^{2}(sin(2\pi y)sin(\pi z)^2 - sin(\pi y)^2sin(2\pi z)),\\
&u_{2}(x,y,z)\\&=sin(\pi y)^{2}(sin(2\pi z)sin(\pi x)^2 - sin(\pi z)^2sin(2\pi x)),\\
&u_{3}(x,y,z)\\&=sin(\pi z)^{2}(sin(2\pi x)sin(\pi y)^2 - sin(\pi x)^2sin(2\pi y)),\\
&p(x,y,z)=sin(\pi x)sin(\pi y)cos(\pi z),
\end{aligned}
\end{equation}
in $\Omega=(0,1)^3$. Then, the right hands $f(x,y)$ and $f(x,y,z)$ can be determined by equation (\ref{stokes-1}), respectively.

In order to demonstrate the effectiveness and accuracy of the DGM, we put forward the norms as follows
\begin{align}
&errL^{1}=\frac{1}{N}\sum_{i=1}^{N}\mid {U}_{i}-{u}_{i}\mid,\label{error1}\\
&errL^{2}=\frac{1}{N}\sum_{i=1}^{N}\mid {U}_{i}-{u}_{i}\mid^{2},\label{error2}\\
&J(\overline{U})=\frac{1}{N}\sum_{i=1}^{N}\Big[\big|\mathcal{G}[\overline{U_i}]\big|^{2}+\big|\nabla\cdot {U_i}\big|^{2}+|U_i-g_i\big|^{2}\Big],\label{loss error}
\end{align}
where $ \overline{U}_{i}$ and ${u}_{i}$ are the neural network and exact solution on each batch $i=1,2, \cdots N$ of datasets, respectively. For simplicity, we only calculate the error of velocity. The pressure is similar.
\begin{figure}[ht]
\centering
\includegraphics[scale=0.4]{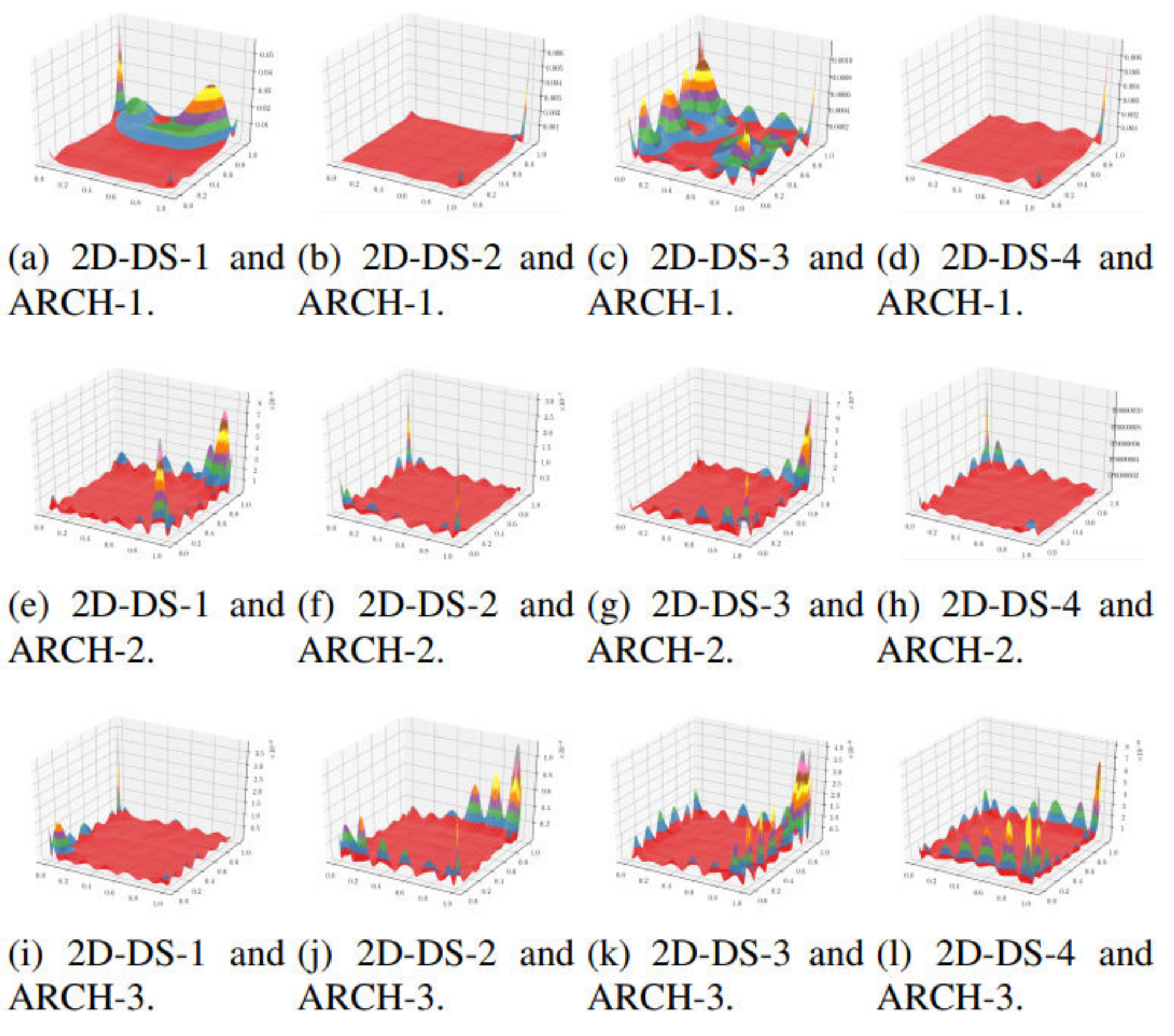}
\caption{The $errL^{2}$ norm of the 2D Stokes equations.}
\label{STOKESL2loss}
\end{figure}

\begin{figure}[ht]
\centering
\includegraphics[scale=0.4]{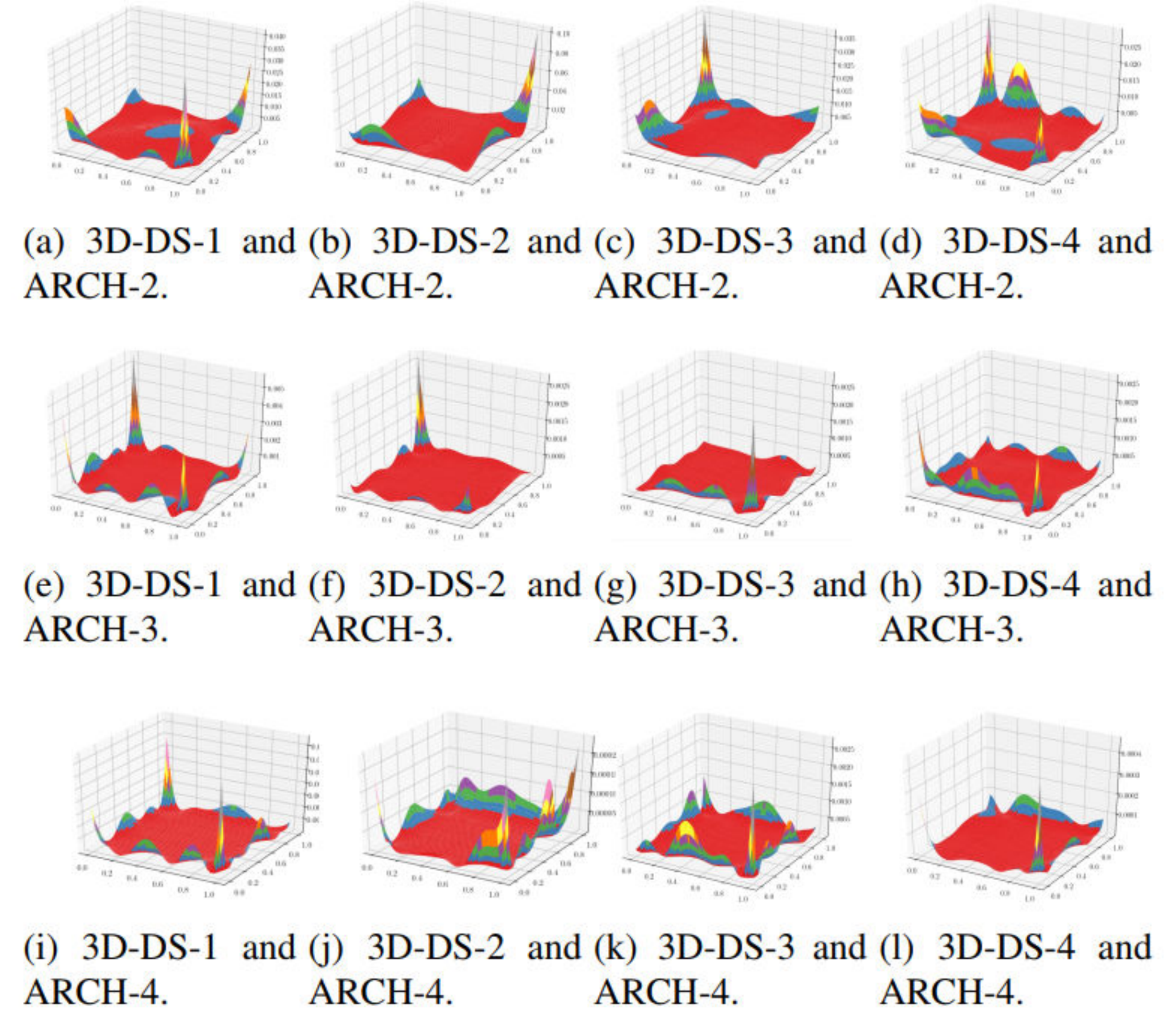}
\caption{The $errL^{2}$ norm of the 3D Stokes equations. }
\label{Fig.3d2}
\end{figure}

\begin{figure}[ht]
\centering
\includegraphics[scale=0.4]{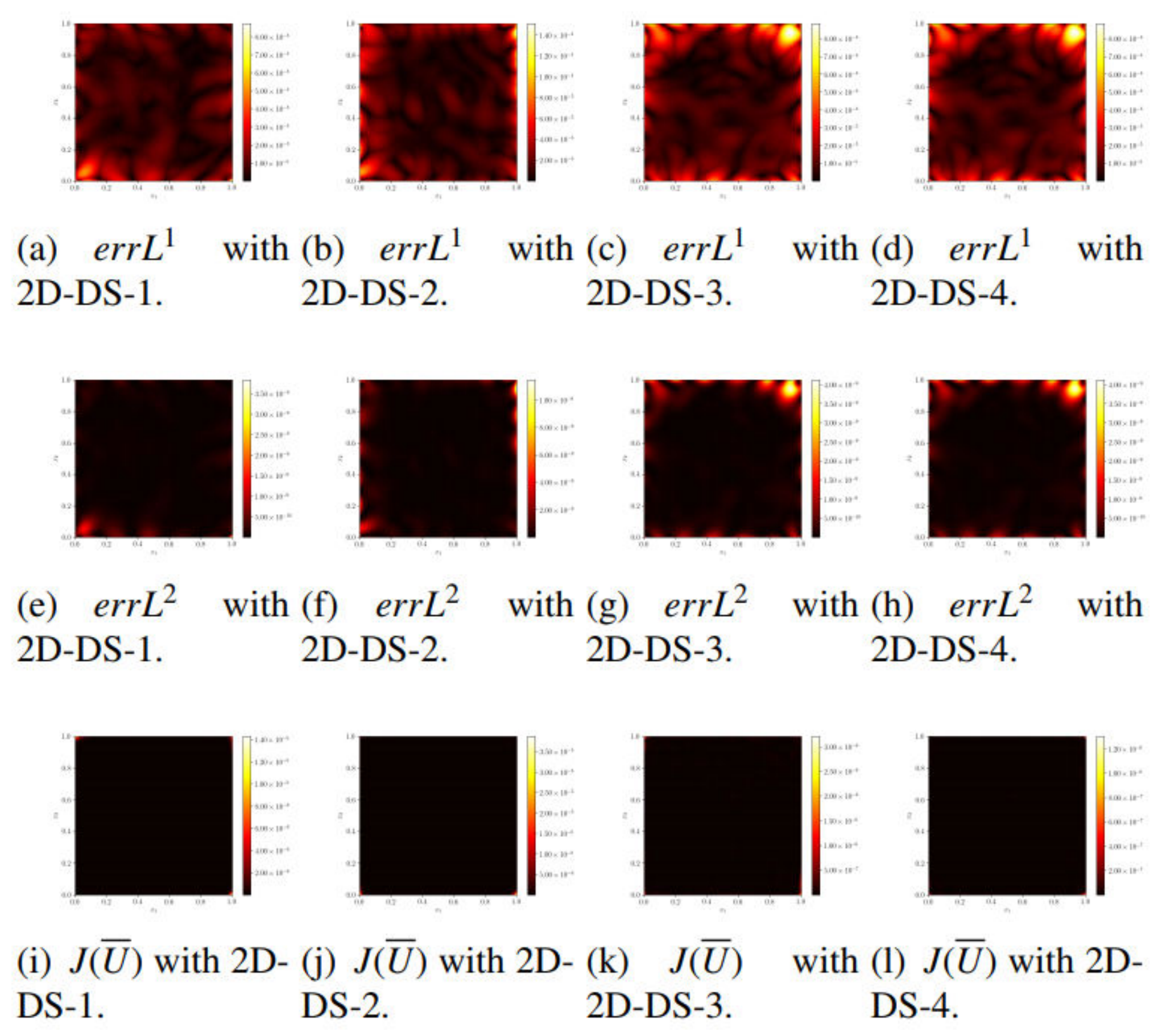}
\caption{The three norms of the 2D Stokes equations by using ARCH-3. }
\label{STOKESA3}
\end{figure}

\begin{table}[ht]
\centering
\caption{The performance of the DGM for the 2D Stokes equations.}
\begin{tabular}{c@{\extracolsep{0.2em}}c@{\extracolsep{0.2em}}c@{\extracolsep{0.2em}}c@{\extracolsep{0.2em}}c@{\extracolsep{0.2em}}c}
\hline
ARCH 1    &2D-DS-1           &2D-DS-2              &2D-DS-3              &2D-DS-4     \\ \hline
$errL^{1}$     &$7.82\times10^{-2}$  &$\bf{7.31\times10^{-3}}$     &$1.41\times10^{-2}$     &$9.39\times10^{-3}$  \\
$errL^{2}$     &$5.48\times10^{-3}$  &$\bf{4.91\times10^{-5}}$     &$1.58\times10^{-4}$     &$8.23\times10^{-5}$  \\
$J(\overline{U})$    &$1.52\times10^{-2}$  &$2.05\times10^{-3}$     &$\bf{1.75\times10^{-3}}$     &$1.79\times10^{-3}$    \\\hline
ARCH 2    &2D-DS-1           &2D-DS-2              &2D-DS-3              &2D-DS-4      \\ \hline
$errL^{1}$     &$5.45\times10^{-5}$  &$5.70\times10^{-5}$     &$\bf{4.18\times10^{-5}}$     &$1.30\times10^{-4}$    \\
$errL^{2}$    &$3.18\times10^{-9}$  &$3.34\times10^{-9}$     &$\bf{1.95\times10^{-9}}$     &$1.60\times10^{-8}$   \\
$J(\overline{U})$   &$\bf{9.60\times10^{-8}}$  &$2.71\times10^{-7}$     &$1.02\times10^{-7}$     &$5.14\times10^{-7}$ \\\hline
ARCH 3   &2D-DS-1           &2D-DS-2              &2D-DS-3              &2D-DS-4   \\ \hline
$errL^{1}$    &$8.09\times10^{-6}$  &$1.53\times10^{-5}$     &$1.23\times10^{-5}$     &$\bf{3.93\times10^{-6}}$ \\
$errL^{2}$    &$6.23\times10^{-11}$ &$2.58\times10^{-10}$    &$1.70\times10^{-10}$    &$\bf{1.91\times10^{-11}}$\\
$J(\overline{U})$   &$1.77\times10^{-8}$  &$4.36\times10^{-8}$     &$1.15\times10^{-8}$     &$\bf{2.59\times10^{-9}}$   \\\hline\
\end{tabular}
\label{lab.121}
\end{table}
\begin{table}[ht]
\centering
\caption{The performance of the DGM for the 3D Stokes equations, when $z=0.5$.}
\begin{tabular}{c@{\extracolsep{0.2em}}c@{\extracolsep{0.2em}}c@{\extracolsep{0.2em}}c@{\extracolsep{0.2em}}c@{\extracolsep{0.2em}}c}
\hline
ARCH 2     &3D-DS-1           &3D-DS-2              &3D-DS-3              &3D-DS-4      \\ \hline
$errL^{1}$    &$6.28\times10^{-2}$         &$1.01\times10^{-1}$     &$5.91\times10^{-2}$     &$\bf{5.82\times10^{-2}}$    \\
$errL^{2}$    &$2.18\times10^{-3}$         &$5.53\times10^{-3}$     &$2.01\times10^{-3}$     &$\bf{1.93\times10^{-3}}$\\
$J(\overline{U})$    &$9.59\times10^{-2}$  &$9.83\times10^{-2}$     &$\bf{5.08\times10^{-2}}$     &$6.26\times10^{-2}$    \\\hline
ARCH 3     &3D-DS-1           &3D-DS-2              &3D-DS-3              &3D-DS-4    \\ \hline
$errL^{1}$    &$1.75\times10^{-2}$         &$\bf{7.68\times10^{-3}}$     &$9.69\times10^{-3}$     &$1.32\times10^{-2}$ \\
$errL^{2}$    &$1.99\times10^{-4}$         &$\bf{4.14\times10^{-5}}$     &$6.34\times10^{-5}$     &$1.00\times10^{-4}$\\
$J(\overline{U})$    &$8.89\times10^{-3}$  &$4.07\times10^{-3}$     &$\bf{2.53\times10^{-3}}$     &$6.90\times10^{-3}$    \\\hline
ARCH 4    &3D-DS-1           &3D-DS-2              &3D-DS-3              &3D-DS-4    \\ \hline
$errL^{1}$    &$1.15\times10^{-2}$         &$4.44\times10^{-3}$     &$9.69\times10^{-3}$     &$\bf{4.23\times10^{-3}}$ \\
$errL^{2}$    &$1.02\times10^{-4}$         &$1.31\times10^{-5}$     &$6.34\times10^{-5}$     &$\bf{1.27\times10^{-5}}$\\
$J(\overline{U})$    &$5.49\times10^{-3}$  &$7.07\times10^{-4}$     &$1.74\times10^{-3}$     &$\bf{6.56\times10^{-4}}$    \\\hline
ARCH 5    &3D-DS-1           &3D-DS-2              &3D-DS-3              &3D-DS-4    \\ \hline
$errL^{1}$    &$2.20\times10^{-3}$         &$\bf{1.60\times10^{-3}}$     &$1.63\times10^{-3}$     &$2.10\times10^{-3}$ \\
$errL^{2}$    &$3.23\times10^{-6}$         &$\bf{1.91\times10^{-6}}$     &$2.38\times10^{-6}$     &$3.77\times10^{-6}$\\
$J(\overline{U})$    &$3.49\times10^{-4}$  &$\bf{8.42\times10^{-5}}$     &$9.56\times10^{-5}$     &$1.01\times10^{-4}$    \\\hline
ARCH 6    &3D-DS-1           &3D-DS-2              &3D-DS-3              &3D-DS-4    \\ \hline
$errL^{1}$    &$1.23\times10^{-3}$         &$1.75\times10^{-3}$     &$-$        &$\bf{9.75\times10^{-4}}$ \\
$errL^{2}$    &$1.33\times10^{-6}$         &$2.34\times10^{-6}$     &$-$
&$\bf{7.10\times10^{-7}}$\\
$J(\overline{U})$    &$2.49\times10^{-4}$  &$1.38\times10^{-4}$     &$-$
&$\bf{3.94\times10^{-5}}$    \\\hline
\end{tabular}
\label{lab.3d2}
\end{table}
Figures \ref{STOKESL2loss} - \ref{Fig.3d2} demonstrate the $errL^{2}$ norm in both 2D and 3D cases ($z=0.5$). Observed from Figures \ref{STOKESL2loss} - \ref{Fig.3d2}, the more red areas, the better performance of the algorithm. Certainly, we find that the subfigures $b$, $d$ in Figure \ref{STOKESL2loss}, $f$, $g$ in Figure \ref{Fig.3d2} are stable than others. The numerical result of subfigure $l$ in Figure \ref{STOKESL2loss} is $1.91\times10^{-11}$, but it does not look stable, which is probably due to the influence of boundary points and corner points. In addition, Figures \ref{STOKESA3} depict $errL^{2}$ norm between $u$ and $U$ in 2D case. Tables \ref{lab.121} - \ref{lab.3d2} display three norms in both 2D and 3D cases ($z=0.5$), the best results for each ARCH are marked out. From Tables \ref{lab.121} - \ref{lab.3d2}, we can find that numerical results with different datasets are less distinguishable. The precision of the neural network is related to the number of layers and neurons, and does not depend on the size of the datasets. The accuracy of neural networks gets better and better only as the number of hidden layers increases. In 2D case, an interesting phenomenon can be found that the best result obtained by using ARCH-3 and the network with more than three hidden layers will have over-fitting. Obviously, for high-dimensional problems, fewer layers neural network is not enough to achieve the required precision. Therefore, it is indispensable to adopt deeper layers since the non-deep neural network has great limitation for the expression of nonlinear relationship. A particularly significant consideration, there appears over-fitting phenomena by using 3D-DS-3 and ARCH-6, which shows that 6 hidden layers is enough to solve the 3D problem.
\begin{figure}[ht]
\centering
\includegraphics[scale=0.4]{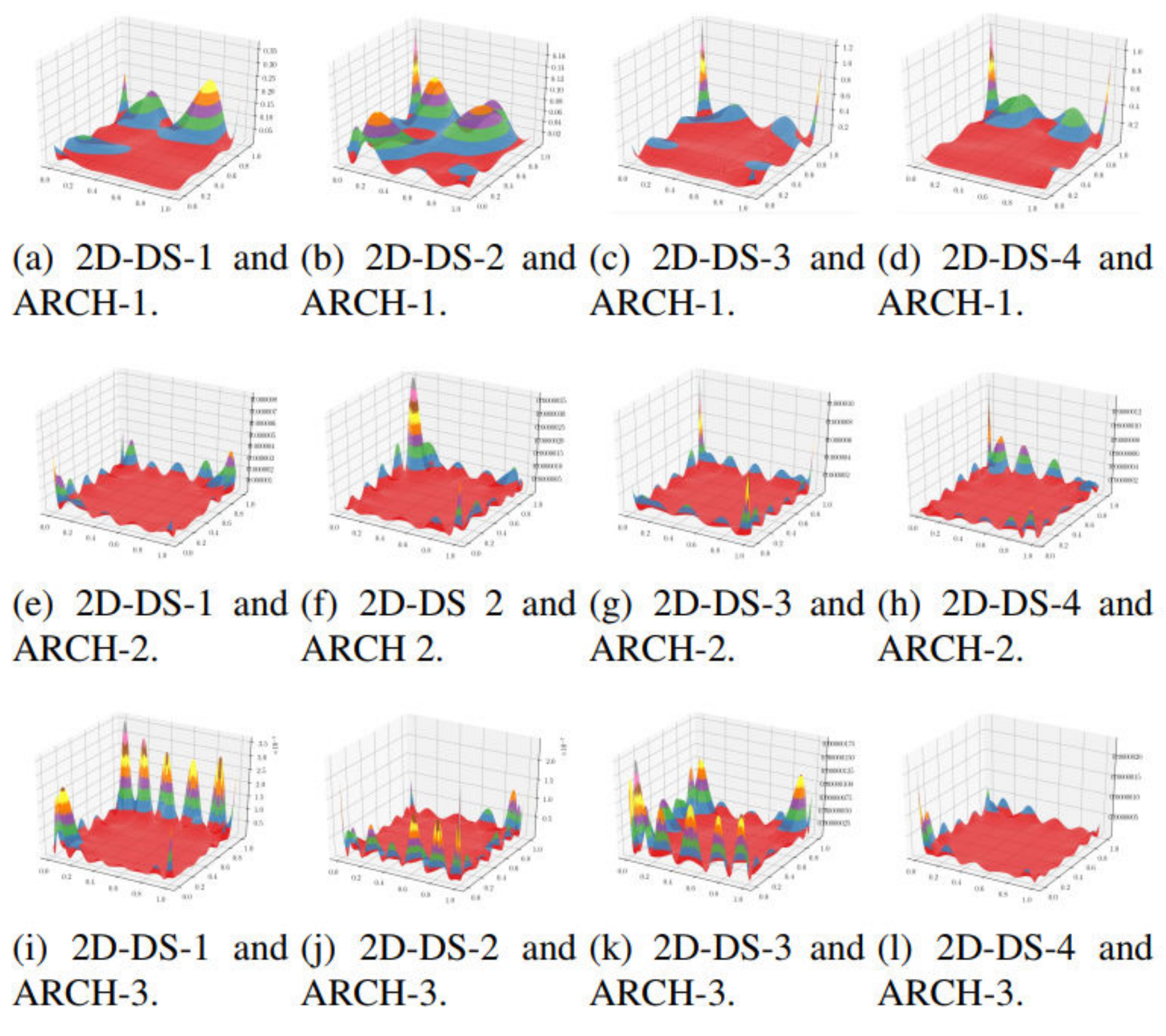}
\caption{The $errL^{2}$ norm of the 2D general Stokes equations. }
\label{GENERALSTOKESL2loss}
\end{figure}

\subsection{The general Stokes equations}
Encouraged by positive results in previous experiments, in this subsection, we mainly consider the effect of the DGM for the general Stokes equations with homogeneous boundary condition in both 2D and 3D cases. Here, we set $\alpha=1$, $\nu=1$ in equations (\ref{stokes-1}) - (\ref{stokes-3}), and apply the analytical solutions (\ref{11}) and (\ref{3dequation1}) for 2D and 3D cases respectively. Consequently, the right hands $f(x,y)$ and $f(x,y,z)$ can be derived by  equation (\ref{stokes-1}).
\begin{figure}[ht]
\centering
\includegraphics[scale=0.4]{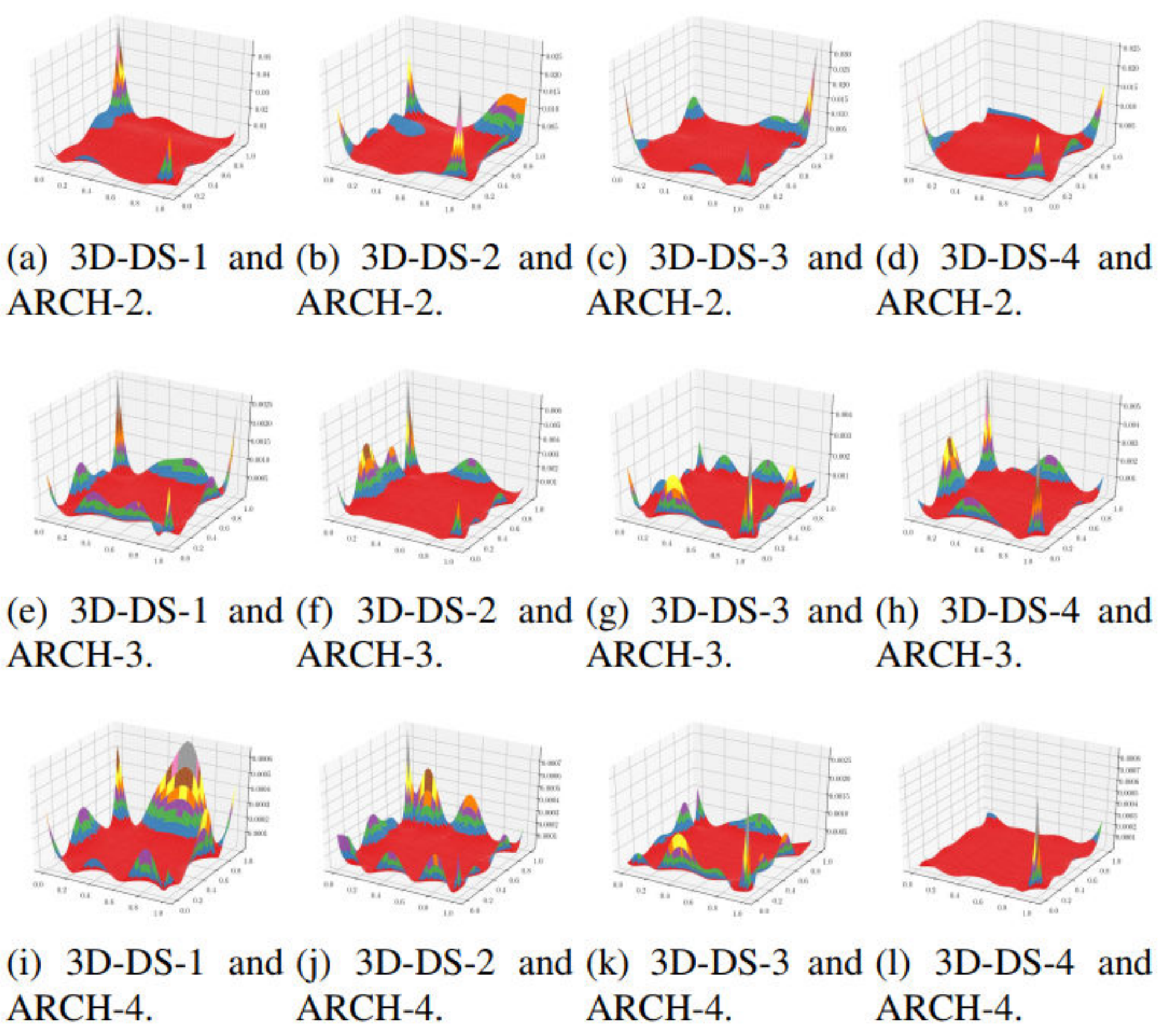}
\caption{The $errL^{2}$ norm of the 3D general Stokes equations. }
\label{Fig.3d3}
\end{figure}
\begin{table}[ht]
\centering
\caption{The performance of the DGM for the 2D general Stokes equations.}
\begin{tabular}{c@{\extracolsep{0.2em}}c@{\extracolsep{0.2em}}c@{\extracolsep{0.2em}}c@{\extracolsep{0.2em}}c@{\extracolsep{0.2em}}c}
\hline
ARCH 1     &2D-DS-1           &2D-DS-2              &2D-DS-3              &2D-DS-4   \\ \hline
$errL^{1}$     &$\bf{2.09\times10^{-1}}$  &$2.12\times10^{-1}$     &$2.84\times10^{-1}$     &$2.97\times10^{-1}$  \\
$errL^{2}$     &$3.65\times10^{-2}$  &$\bf{3.05\times10^{-2}}$     &$7.02\times10^{-2}$     &$7.87\times10^{-2}$  \\
$J(\overline{U})$        &$1.4\times10^{-1}$  &$\bf{7.74\times10^{-2}}$     &$2.38\times10^{-1}$      &$2.30\times10^{-1} $    \\\hline
ARCH 2     &2D-DS-1           &2D-DS-2              &2D-DS-3              &2D-DS-4      \\ \hline
$errL^{1}$    &$4.70\times10^{-4}$  &$3.86\times10^{-4}$     &$4.98\times10^{-4}$     &$\bf{1.81\times10^{-4}}$    \\
$errL^{2}$    &$2.13\times10^{-7}$  &$1.54\times10^{-7}$     &$2.43\times10^{-7}$     &$\bf{3.69\times10^{-8}}$   \\
$J(\overline{U})$        &$1.68\times10^{-5}$  &$1.70\times10^{-5}$     &$9.57\times10^{-6} $    &$\bf{2.06\times10^{-6}} $ \\\hline
ARCH 3     &2D-DS-1           &2D-DS-2              &2D-DS-3              &2D-DS-4   \\ \hline
$errL^{1}$    &$1.25\times10^{-4}$  &$\bf{8.99\times10^{-5}}$     &$3.71\times10^{-4}$     &$1.88\times10^{-4}$ \\
$errL^{2}$    &$1.78\times10^{-8}$  &$\bf{8.08\times10^{-9}}$     &$1.33\times10^{-7}$     &$3.79\times10^{-8}$\\
$J(\overline{U})$        &$3.04\times10^{-6}$  &$1.94\times10^{-6}$     &$2.20\times10^{-6}$     &$\bf{1.48\times10^{-6}}$   \\\hline
\end{tabular}
\label{lab.1211}
\end{table}
The plots of $errL^{2}$ norm are shown in Figures \ref{GENERALSTOKESL2loss} - \ref{Fig.3d3}. In the same manner, the value of three different norms between the neural network and the exact solution are displayed in Tables \ref{lab.1211} - \ref{lab.3d3}. Evidently, the $errL^{2}$ norm is going to change by using different ARCHs, as shown in Figures \ref{GENERALSTOKESL2loss} - \ref{Fig.3d3} and Tables \ref{lab.1211} - \ref{lab.3d3}. The similar results just as for the 2D case can be derived by Figures \ref{GENERALSTOKESA3}. In the same way, the error decreases gradually as the number of hidden layers increases, especially by using ARCH-3 and ARCH-6 in 2D and 3D case respectively.
\begin{figure}[ht]
\centering
\includegraphics[scale=0.4]{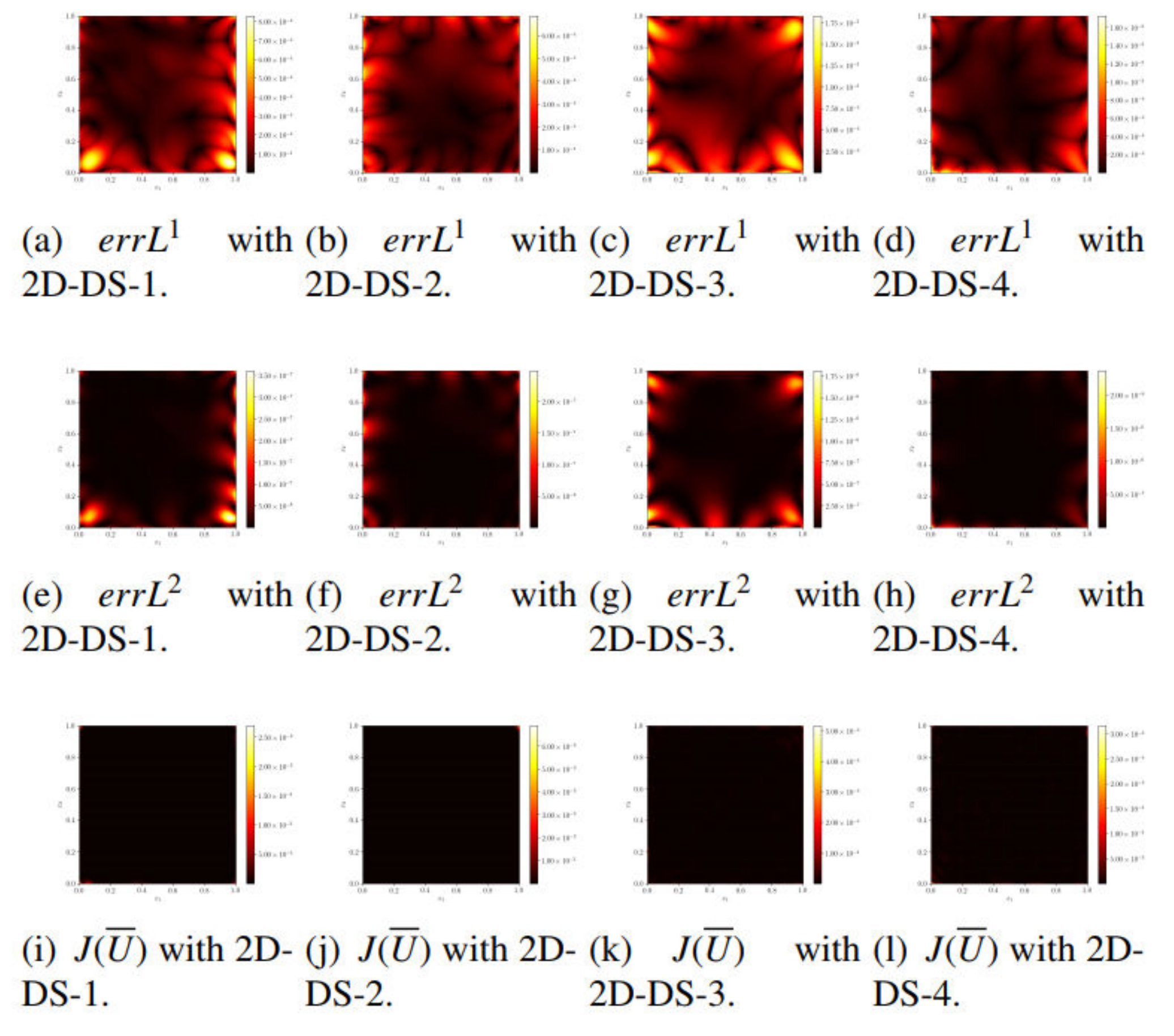}
\caption{The three norms of the 2D general Stokes equations by using ARCH-3. }
\label{GENERALSTOKESA3}
\end{figure}

\begin{table}[ht]
\centering
\caption{The performance of the DGM for the 3D general Stokes equations, when $z=0.5$.}
\begin{tabular}{c@{\extracolsep{0.2em}}c@{\extracolsep{0.2em}}c@{\extracolsep{0.2em}}c@{\extracolsep{0.2em}}c@{\extracolsep{0.2em}}c}
\hline
ARCH 2     &3D-DS-1           &3D-DS-2              &3D-DS-3              &3D-DS-4     \\ \hline
$errL^{1}$    &$6.73\times10^{-2}$         &$6.19\times10^{-2}$     &$4.67\times10^{-2}$     &$\bf{4.15\times10^{-2}}$    \\
$errL^{2}$    &$2.52\times10^{-3}$         &$1.89\times10^{-3}$     &$1.26\times10^{-3}$ &$\bf{1.06\times10^{-3}}$\\
$J(\overline{U})$        &$6.36\times10^{-2}$  &$4.89\times10^{-2}$     &$\bf{4.69\times10^{-2}} $    &$5.16\times10^{-2} $ \\\hline
ARCH 3    &3D-DS-1           &3D-DS-2              &3D-DS-3              &3D-DS-4  \\ \hline
$errL^{1}$    &$\bf{1.38\times10^{-2}}$         &$1.80\times10^{-2}$     &$1.53\times10^{-2}$     &$1.52\times10^{-2}$ \\
$errL^{2}$    &$\bf{1.30\times10^{-4}}$         &$2.29\times10^{-4}$     &$1.67\times10^{-4}$     &$1.81\times10^{-4}$\\
$J(\overline{U})$        &$8.96\times10^{-3}$  &$\bf{3.17\times10^{-3}}$     &$4.19\times10^{-3} $    &$3.39\times10^{-3} $ \\\hline
ARCH 4    &3D-DS-1           &3D-DS-2              &3D-DS-3              &3D-DS-4  \\ \hline
$errL^{1}$    &$7.52\times10^{-3}$         &$6.48\times10^{-3}$     &$1.16\times10^{-2}$     &$\bf{3.58\times10^{-3}}$ \\
$errL^{2}$    &$4.19\times10^{-5}$         &$3.08\times10^{-5}$     &$9.77\times10^{-5}$     &$\bf{9.11\times10^{-6}}$\\
$J(\overline{U})$        &$1.02\times10^{-2}$  &$6.34\times10^{-4}$     &$1.63\times10^{-3} $    &$\bf{4.51\times10^{-4}} $ \\\hline
ARCH 5    &3D-DS-1           &3D-DS-2              &3D-DS-3              &3D-DS-4  \\ \hline
$errL^{1}$    &$\bf{1.96\times10^{-3}}$         &$2.67\times10^{-3}$     &$8.78\times10^{-3}$
&$-$ \\
$errL^{2}$    &$\bf{3.00\times10^{-6}}$         &$4.53\times10^{-6}$     &$5.84\times10^{-5}$     &$-$\\
$J(\overline{U})$        &$5.15\times10^{-4}$  &$\bf{2.08\times10^{-4}}$     &$6.90\times10^{-4} $    &$-$ \\\hline
ARCH 6    &3D-DS-1           &3D-DS-2              &3D-DS-3              &3D-DS-4  \\ \hline
$errL^{1}$    &$1.66\times10^{-3}$         &$1.87\times10^{-3}$     &$\bf{1.28\times10^{-3}}$
&$2.11\times10^{-3}$ \\
$errL^{2}$    &$2.37\times10^{-6}$         &$2.63\times10^{-6}$     &$\bf{1.60\times10^{-6}}$     &$3.31\times10^{-6}$\\
$J(\overline{U})$        &$1.21\times10^{-4}$  &$1.20\times10^{-4}$     &$\bf{7.37\times10^{-5}} $    &$1.16\times10^{-4}$ \\\hline
\end{tabular}
\label{lab.3d3}
\end{table}

\subsection{The driven cavity flow}
The driven cavity flows have been extensively applied as test cases for validating the incompressible fluid dynamics algorithm. The corner singularities for the 2D fluid flows are very important since most examples of physical interest have corners. In these two examples, we consider the 2D driven flow in a rectangular cavity when the top surface moves with a constant velocity along its length. The upper corners where the moving surface meets the stationary walls are singular points of the flow at the multi-valued horizontal velocity. The lower corners are also weakly singular points. Moreover, we also consider the 3D driven flow in a cube of unit volume, centered at $x = y = z = 0.5$ (Figure \ref{23dflow}). A unit tangential velocity in the $x$ direction is prescribed at the top surface, while zero velocity is prescribed on the remaining bounding surfaces in numerical examples.
\begin{figure}[H]
\centering
\includegraphics[scale=0.4]{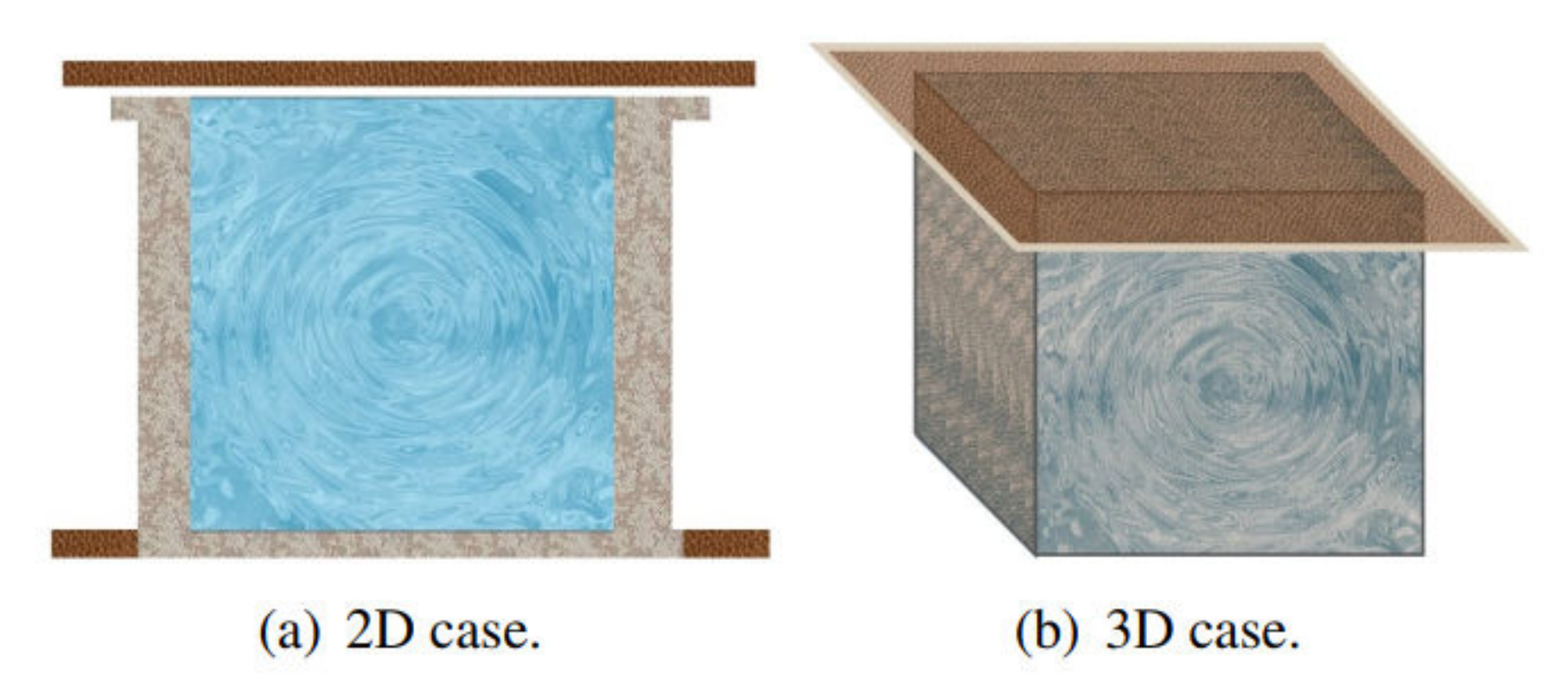}
\caption{Diagram of 2D and 3D square cavity flow.}
\label{23dflow}
\end{figure}

\begin{figure}[ht]
\centering
\includegraphics[scale=0.4]{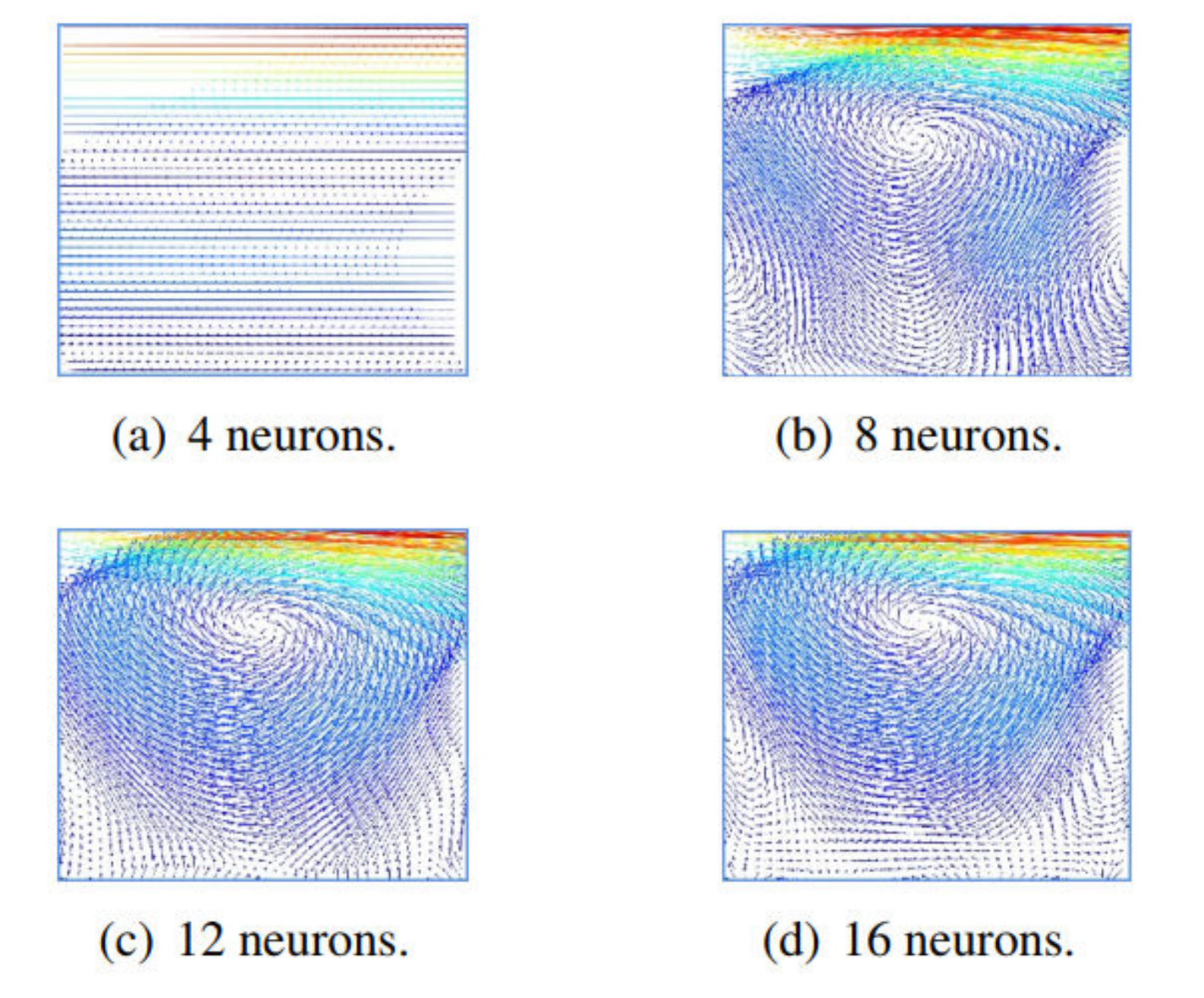}
\caption{The 2D driven cavity flow with different neurons and ARCH-1.}
\label{Fig.12}
\end{figure}

\begin{figure}[ht]
\centering
\includegraphics[scale=0.4]{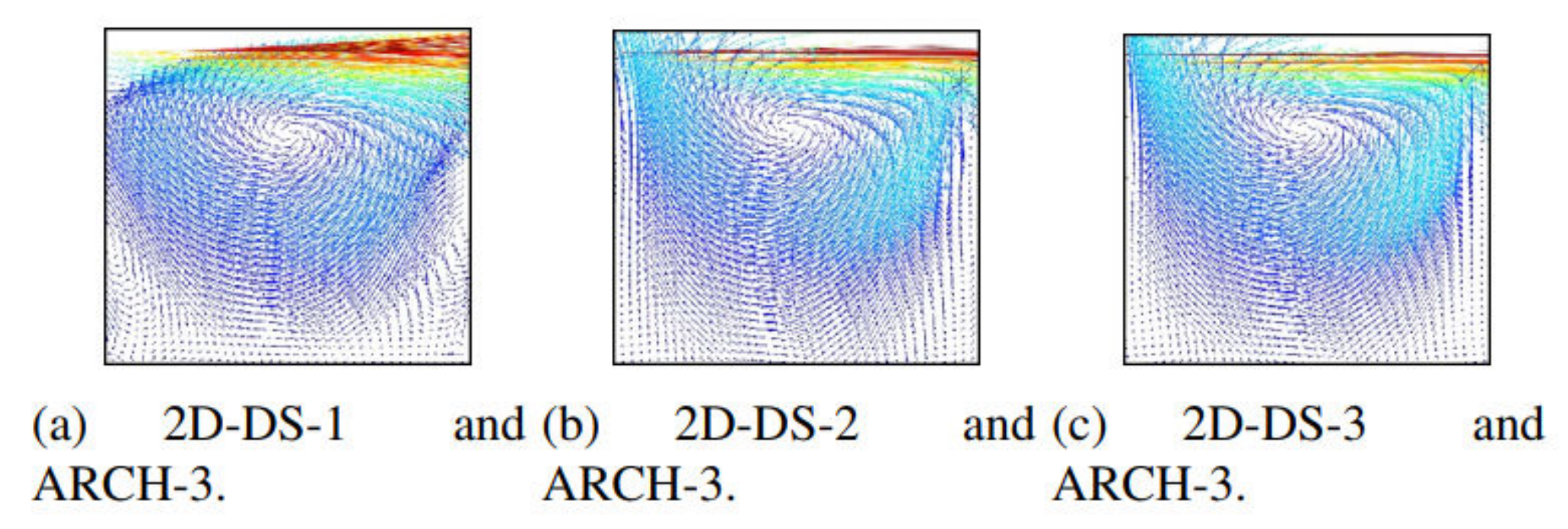}
\caption{The 2D driven cavity flow. }
\label{Fig.13}
\end{figure}

\begin{figure}[ht]
\centering
\includegraphics[scale=0.4]{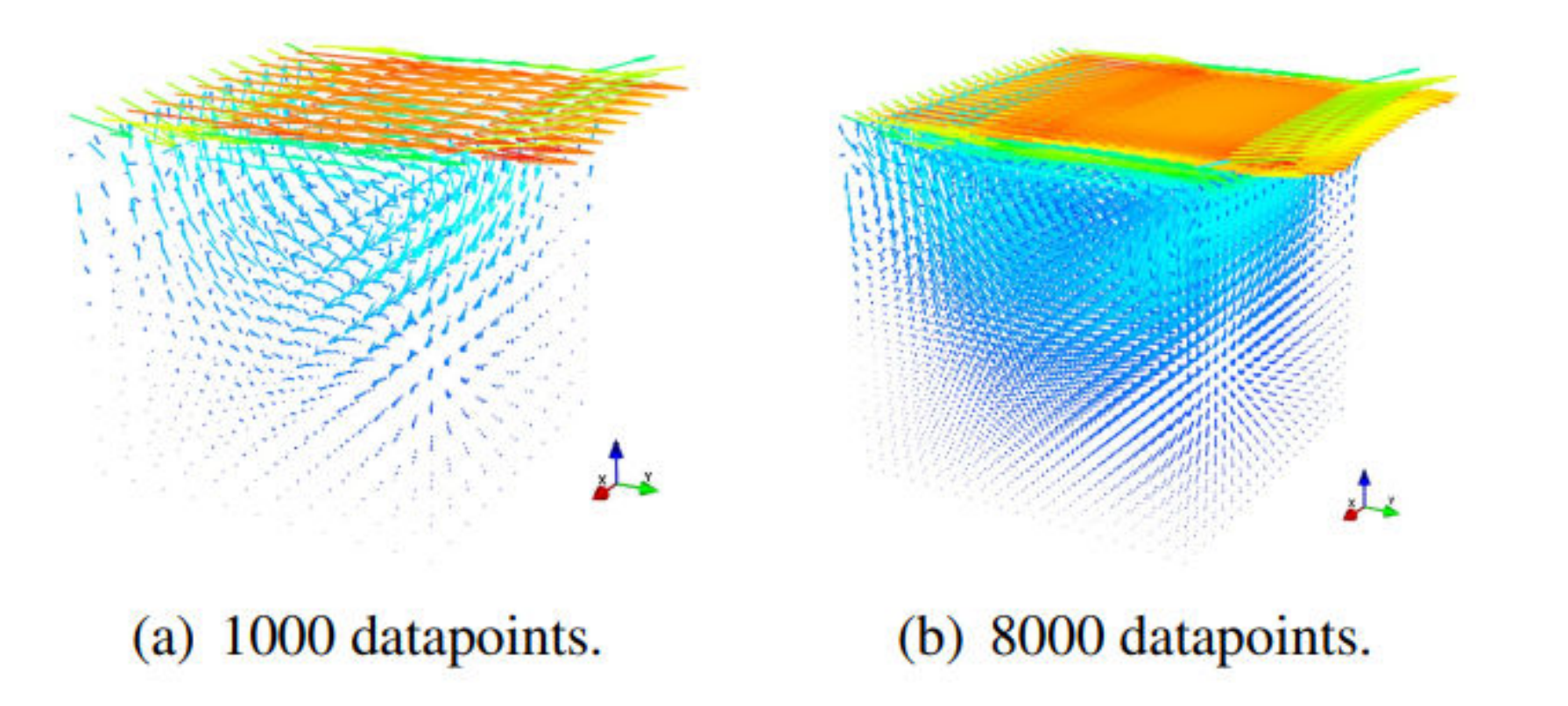}
\caption{The 3D driven cavity flow testing on two different sampled data points.}
\label{Fig.3d4}
\end{figure}

\begin{figure}[ht]
\centering
\includegraphics[scale=0.4]{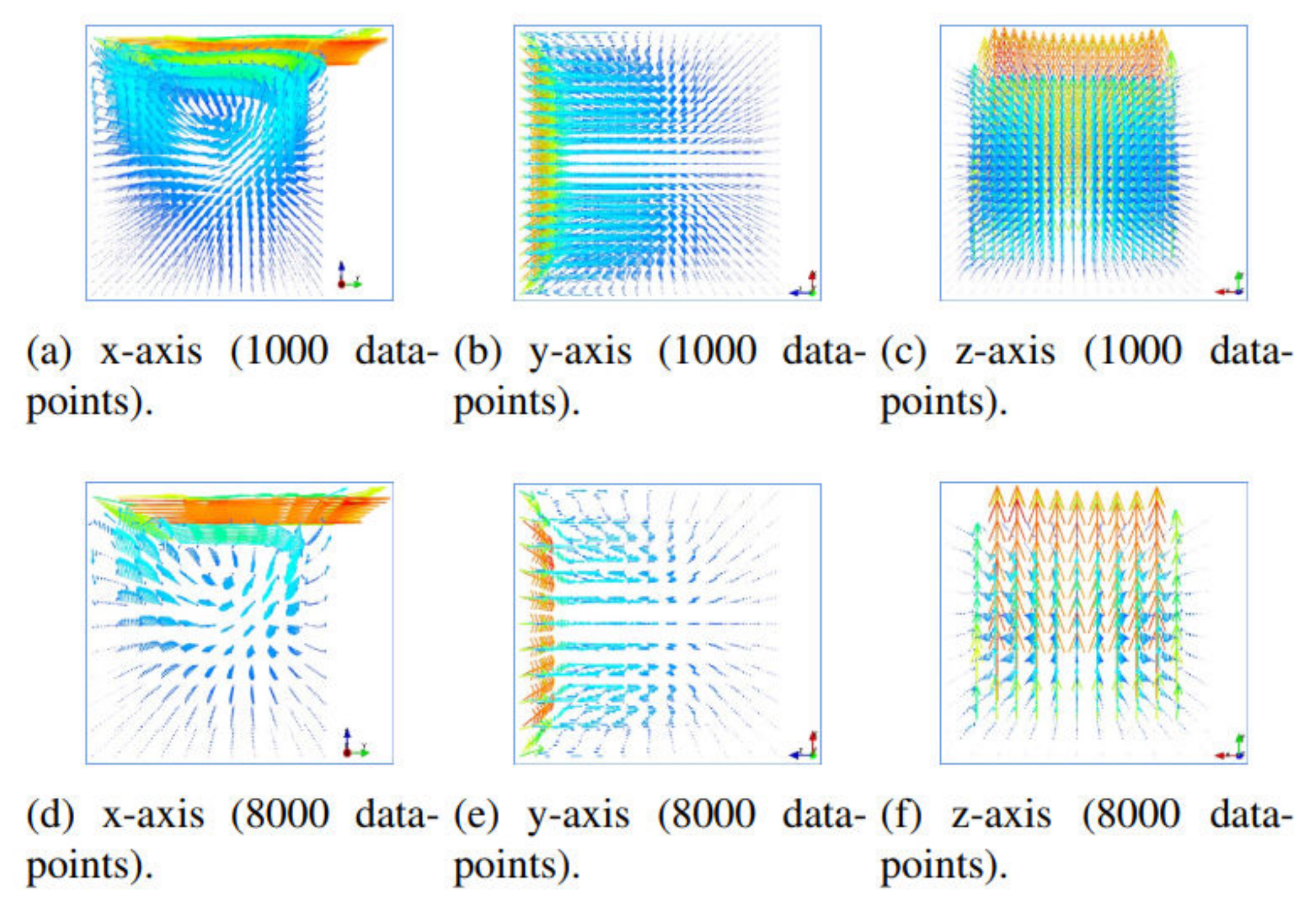}
\caption{The top view along 3 axis of the 3D driven cavity flow. }
\label{Fig.3d5}
\end{figure}
In order to explore how many neurons are enough to simulate the driven cavity flow, we apply 4, 8, 12 and 16 neurons for testing. From Figure \ref{Fig.12}, we can find that the 2D driven cavity flow is stablest when using 16 neurons. Moreover, the optimal models for 2D and 3D cases are obtained by training the existing 2D-DS and 3D-DS respectively. Then, we test the different neural networks by using 1600 data points in 2D case (Figure \ref{Fig.13}). In like wise, the better results are only related to the number of hidden layers. Especially, the results by using ARCH 3 are in perfect agreement with the physical significance since ARCH 3 has more hidden layers than others. Furthermore, the optimal framework for the 3D case obtained by using 3D-DS-3 and ARCH-2. As shown in Figures \ref{Fig.3d4} - \ref{Fig.3d5}, we utilize $1000$ and $8000$ data points for testing and intercepts the top view along 3 axis respectively. The numerical results indicate that the DGM is efficient and accurate.

\section{Conclusions}
This paper applies the DGM to solve the general Stokes problems in both 2D and 3D cases with high efficiency and accuracy, which can transform the traditional grid mesh method into a grid free algorithm by using the random sampled data. Besides, we set the objective function appropriately to convert the constrained problem into an unconstrained problem in the sense and give two theorems to ensure the convergence of the objective function and the convergence of the neural network to the exact solution. In general, this method is based on drawing random sampled points from the domain, which can be readily
extended to arbitrary domains; triangulation of the domain is not needed. The numerical results fully demonstrate the convergence properties of the DGM completely. But, we need more deliberation on the elements which have great impact in experiments. For example, how to construct the most suitable objective function for measuring loss and applied to optimization, whether the deeper network can get better results and randomness have a positive effect on the algorithm.

\end{document}